\documentclass[12pt, reqno]{amsart}

\usepackage{amsmath,amssymb,amsthm,mathtools,wasysym,calc,verbatim,enumitem,tikz,pgfplots,url,mathrsfs,bbm}
\usepackage[left=1in,right=1in,top=1.2in,bottom=1in]{geometry}
\usepackage{latexsym}

\usepackage{color}

\usepackage{amssymb}
\usepackage{verbatim}

\usepackage{bbm}
\usepackage{amsthm}
\usepackage{graphicx}
\usepackage{float}
\usepackage{soul}

\usepackage[colorlinks=true,linkcolor=black ,citecolor=black]{hyperref}

\usepackage{amsthm}

\theoremstyle{plain}
\newtheorem{theorem}{Theorem}[section]
\newtheorem{lemma}[theorem]{Lemma}

\newtheorem*{claim*}{Claim}
\newtheorem{cor}[theorem]{Corollary}
\newtheorem{claim}[theorem]{Claim}
\newtheorem{prop}[theorem]{Proposition}

\theoremstyle{remark}
\newtheorem{defn}[theorem]{Definition}

\newtheorem{question}[theorem]{Question}
\newtheorem{remark}[theorem]{Remark}

\newcommand{\E}{\mathbf{E}}
\newcommand{\Prob}{{\mathbf{P}}}
\newcommand{\rtt}{\mathbf{0}}
\newcommand{\TFM}{{\sf TFM}}
\newcommand{\ATFM}{{\sf ATFM}}
\newcommand{\FM}{{\sf FM}}

\newcommand{\Poiss}{\mathrm{Poiss}}
\newcommand{\SRW}{{\sf SRW}}
\newcommand{\SRWS}{{\sf SRWs}}
\newcommand{\GW}{{\sf GW}}
\newcommand{\AGW}{{\sf AGW}}
\newcommand{\TPPT}{{\sf TreePathParticlesTrajectories}}
\newcommand{\TP}{{\sf TreePath}}
\newcommand{\F}{\mathcal{F}}

\newcommand{\one}{\mathbf{1}}

\newcommand{\HARM}{{\sf HARM} }
\newcommand{\cE}{\mathcal{E}}
\newcommand{\RED}{}
\newcommand{\BLUE}{}
\renewcommand{\epsilon}{\varepsilon}
\newcommand{\eps}{\varepsilon}
\newcommand{\bT}{\mathbf{T}}
\renewcommand{\P}{\mathbf{P}}

\newcommand{\Par}{\overleftarrow}

\begin{document}

	\title[The frog model on Galton-Watson trees]{The frog model on Galton-Watson trees}
	
	\author{Marcus Michelen}
	\address{\tiny{Department of Mathematics, Statistics, and Computer Science, University of Illinois, Chicago}}
	\email{michelen@uic.edu, michelen.math@gmail.com}
	\author{Josh Rosenberg}
	\address{}
	\email{joshrsix@hotmail.com}

		\begin{abstract}
			We consider an interacting particle system on trees known as the frog model: initially, a single active particle begins at the root and i.i.d.~$\mathrm{Poiss}(\lambda)$ many inactive particles are placed at each non-root vertex.  Active particles perform discrete time simple random walk and activate the inactive particles they encounter.  We show that for Galton-Watson trees with offspring distributions $Z$ satisfying $\Prob(Z \geq 2) = 1$ and $\E[Z^{4 + \epsilon}] < \infty$ for some $\epsilon > 0$, there is a critical value $\lambda_c\in(0,\infty)$ separating recurrent and transient regimes for almost surely every tree, thereby answering a question of Hoffman-Johnson-Junge.  In addition, we also establish that this critical parameter depends on the entire offspring distribution, not just the maximum value of $Z$, answering another question of Hoffman-Johnson-Junge and showing that the frog model and contact process behave differently on Galton-Watson trees.
		\end{abstract}

	\maketitle 
	\section{Introduction}
	
	The frog model refers to a particular kind of system of interacting random walks on a rooted graph.  In its initial state, it features a single active particle at the root, and some collection of inactive particles distributed 
	among the non-root vertices.  The active particle at the root begins performing a discrete time simple random walk on the graph, and any time an active particle lands on a vertex containing inactive particles, they all become activated and begin performing their own independent discrete time simple random walks, activating any sleeping particles that they encounter along the way.  The particles in this system are often referred to as frogs, with active particles deemed ``awake'' and inactive particles ``sleeping.''
	
	The frog model is a very simple model of the spread of infection.  Additionally, it bears resemblance to activated random walk---a fundamental and well-studied interacting particle system on graphs (see the survey \cite{rolla} for more context)---in which active particles later become inactive.  As such, the frog model may be understood as a more active version of activated random walk, and indeed can be used to stochastically dominate activated random walk.
	
	On infinite graphs, studies of the frog model have often focused on determining whether it is recurrent (meaning almost surely infinitely many active particles hit the root) or transient (meaning almost surely only finitely many active particles ever hit the root).  Much work has been done on the frog model on $\mathbb{Z}^d$: in \cite{TW}, Telcs and Wormald showed that the one frog per site model on $\mathbb{Z}^d$ is recurrent for every $d\geq 1$.  This was extended in \cite{alves} to show recurrence for any i.i.d. configuration of frogs on $\mathbb{Z}^d$.  In order to obtain transitions from recurrence to transience on $\mathbb{Z}^d$, one may take the density of frogs to be non-uniform \cite{popov-random} or bias the walks in a given direction \cite{DP,gantert-schmidt,GNR}.
	
	On trees, the story is quite different: in a breakthrough work \cite{HJJ2}, Hoffman, Johnson and Junge  demonstrated that the one-frog-per-vertex model is recurrent on the $d$-ary tree for $d=2$, and transient for $d\geq 5$; the cases of $d=3,4$ remain open.  Likewise, in \cite{HJJ1}, Hoffman, Johnson, and Junge showed that if $\Poiss(\lambda)$ sleeping frogs are placed at each vertex on a $d$-ary tree, recurrent and transient regimes may be found as $\lambda$ varies for each $d \geq 2$.  They comment, ``we believe that the most interesting aspect of this work is that the frog model on trees is teetering on the edge between recurrence and transience.''
	
	In the case of both the one frog per-site model on the $d$-ary tree studied in \cite{HJJ2}, as well as the i.i.d. $\text{Poiss}(\lambda)$ frogs per-site version from \cite{HJJ1}, the proofs of transience employed a reasonably straightforward approach that involved embedding the model in a more standard branching random walk model, and then establishing transience of the branching random walk via a martingale argument.  Conversely, the proofs of recurrence from both \cite{HJJ2} and \cite{HJJ1} tended to require significantly more ingenuity, featuring clever bootstrapping arguments that rely heavily on the self-similarity properties of regular trees.  More generally, studies of the frog model on infinite graphs have almost always involved graphs with near-perfect symmetry.  While there have been efforts to break free of this constraint, they have been limited in scope.  One such attempt was made by the second author in \cite{Rosenberg}, where he established recurrence for the one per-site model on the $3,2$-alternating tree.  Yet even in this case, the proof exploits the self-similarity of the bi-regular tree, specifically relying on the fact that (as with all of the other graphs referenced above) the $3,2$-alternating tree is quasi-transitive, meaning that the set of vertices can be partitioned into finitely many sets so that for each pair of vertices in the same set there exists a graph automorphism mapping one to the other.
	
	\subsection{Results}
	
	In the present work we examine the frog model with i.i.d. $\text{Poiss}(\lambda)$ sleeping frogs positioned at each non-root vertex of a Galton-Watson tree.  Our main results, which are contained in the following theorem, show that there is a sharp transition from transience to recurrence provided the offspring distribution is always at least $2$ and has sufficiently many moments, while also establishing an asymptotic upper bound on the value of the critical parameter $\lambda_c$.
	
	\medskip
	\begin{theorem}\label{theorem:mrecthjg}
		Let ${\sf GW}$ be the measure on Galton-Watson trees induced by an offspring distribution $Z$ for which $\Prob(Z\geq 2)=1$ and $\E[Z^{4+\epsilon}]<\infty$ for some $\epsilon>0$.  Then there exists a constant $\lambda_c\in(0,\infty)$ such that, for ${\sf GW}$-a.s.\ every {\BLUE tree} ${\bf T}$, the frog model with i.i.d.\ $\mathrm{Poiss}(\lambda)$ frogs per non-root vertex is transient for every $\lambda<\lambda_c$, and recurrent for every $\lambda>\lambda_c$.  Furthermore, the critical parameter $\lambda_c$ satisfies the bound $\log \lambda_c = O(\epsilon^{-1} \log \E Z^{4 + \epsilon} + \epsilon^{-2}  )$ for $\epsilon \in (0,1)$.
	\end{theorem}
	
	\medskip
	Theorem \ref{theorem:mrecthjg} answers a question posed by Hoffman, Johnson, and Junge in \cite{HJJ1} that involved asking whether or not their recurrence and transience results for the frog model on regular trees can be extended to the Galton-Watson case.  A further question in \cite{HJJ1} asks if recurrence on Galton-Watson trees depends on the entire degree distribution or only the maximal degree; the upper bound on $\lambda_c$ stated in Theorem \ref{theorem:mrecthjg} is likely far from optimal, but is good enough to show that recurrence must in fact depend on the entire degree distribution, rather than just the maximal degree (Corollary \ref{cor:max-degree-example}).  This is in stark contrast to the contact process, where the critical probability for local survival on a Galton-Watson tree depends only on the maximum degree \cite[Proposition $2.5$]{pemantle-stacey}.  
	
	The two most difficult parts of proving Theorem \ref{theorem:mrecthjg} are showing the existence of a recurrent regime, along with proving that no intermediate regime may exist: in particular, we show in Theorem \ref{th:0-1-GW} that $\GW$-a.s., infinitely many particles hit the root with probability either $0$ or $1$.  We also show in Section \ref{ss:no01} that there is no hope of upgrading this $0$-$1$ law for \emph{every} possible tree.  As such, the probabilistic uniformity of Galton-Watson trees plays a crucial role in our proof.  In the case of proving recurrence, we take inspiration from \cite{HJJ1,HJJ2} and replace the trajectory of each particle with its loop-erased version.  After this, our proof diverges radically from the methods of \cite{HJJ1,HJJ2}: whereas in their case the self-similarity of the $d$-ary tree allowed them to work on an extremely local level, there is no opportunity to do so on Galton-Watson trees.  Further, the behavior of random walk on such trees very much depends on the particular instance of the tree---as opposed to on regular trees, in which random walk exhibits a high degree of symmetry---and thus rather than taking the approach of examining the number of particles that visit a given vertex, we instead weigh each activated particle roughly according to the probability that it visits the root.  
	
	We begin the process of proving Theorem \ref{theorem:mrecthjg} in Section \ref{sec:0-1}, where we establish the aforementioned $0$-$1$ law for the frog model on Galton-Watson trees which allows us to rule out the possibility of a non-trivial intermediate phase between transience and recurrence for {\RED almost}-every {\RED Galton-Watson tree} ${\bf T}$.  To prove this result we first focus on augmented Galton-Watson trees, using an ergodic theory argument and  {\RED decomposition} from \cite{LPP} and \cite[Chapter 17]{LP}, in order to establish a $0$-$1$ law that applies ${\sf AGW}$-a.s. {\BLUE (where ${\sf AGW}$ refers to the measure for augmented Galton-Watson trees)}.  We then establish the desired result by showing that recurrence for {\RED almost-every \emph{augmented} Galton-Watson tree} implies {\RED recurrence for almost-every Galton-Watson tree}.
	
	For the proof of transience in Section \ref{sec:transience}, which in fact applies to \emph{every} tree {\BLUE $T$} generated by our offspring distribution $Z$, we essentially adapt the approach that was employed by Hoffman, Johnson, and Junge in \cite{HJJ1} to prove the existence of a transient regime on the regular $d$-ary tree.  This technique involves first coupling the Poisson frog model on $T$ with branching random walk.  We then introduce a weight function that allows us to construct a supermartingale out of this branching random walk model.  This is then used to show that for sufficiently small Poisson mean $\lambda$, the branching random walk model is transient on $T$, which by virtue of stochastic dominance, implies that the original frog model is as well.
	
	The proof of recurrence is more challenging, and occupies Sections \ref{sec:tools}, \ref{sec:proofmain} and \ref{sec:proof-prop}.  In contrast to \cite{HJJ1,HJJ2}, rather than bootstrapping the number of visits to the root, we instead bootstrap the conditional probability that a randomly selected vertex is visited given that its parent is visited.  Further, the recurrence result need not hold on \emph{every} Galton-Watson tree, only on \emph{almost-every} Galton-Watson tree, and so we therefore work with an annealed probability distribution that incorporates the randomness of both the tree and the frog model simultaneously.  The key step in the bootstrapping argument is to focus not on the \emph{number} of particles that are awakened, but rather on the \emph{harmonic measure} of the set of vertices visited on each level of the tree.  The reason for this is that the harmonic measure and return probabilities are typically roughly comparable on Galton-Watson trees (see Lemma \ref{lem:crphmjg} for a precise statement).  Thus, knowing that each vertex in a set with large harmonic measure is visited will imply that a large number of awakened particles will return to the root. Ultimately, this technique then allows us to show that, for sufficiently large Poisson mean $\lambda$, all non-root vertices in our randomly generated tree {\BLUE ${\bf T}$} are activated with probability $1$, from which recurrence follows easily.  
	
	Section \ref{sec:tools} provides three sets of preliminary tools for our proof of recurrence: we will need to compare the harmonic measure with hitting probabilities of random walk; we will need to bound the probability that Galton-Watson trees have a large portion of their harmonic measure given by vertices with many children; and finally we will need to compare different measures on random trees.  
	
	In Section \ref{sec:proofmain} we first introduce two variants of the frog model, which we call the truncated frog model and the augmented truncated frog model. The truncated frog model is constructed by altering the dynamics of the random walks performed by activated particles; most significantly, in the truncated frog model, particles perform \emph{loop-erased random walk} rather than simple random walk, following the lead of \cite{HJJ1,HJJ2}.  We then show that the two models can be coupled in such a way that the number of returns to the root in the ordinary model stochastically dominates that of the truncated frog model (Lemma \ref{lem:truncated-frog}).   Once we have begun working with the truncated frog model, we will want to understand the probability that a vertex is activated given that its parent is activated.  To analyze this event, we introduce the augmented truncated frog model, which is tailored specifically to understand the distribution of the truncated frog model in the subtree $T(v)$ given that $v$ is activated (see Lemma \ref{lem:augmented}).  The main technical engine in our proof of recurrence is a bootstrapping argument for the augmented truncated frog model stated as Proposition \ref{pr:iteratejg}.  In Section \ref{sec:proofmain} we deduce our statement of recurrence (Theorem \ref{theorem:mrecthjg}) from Proposition \ref{pr:iteratejg}.    
	
	In Section \ref{sec:proof-prop} we prove Proposition \ref{pr:iteratejg}, which is the most technical part of the paper. We provide a detailed sketch of the proof in Section \ref{sec:jgsm1}.
	

	The paper concludes with Section \ref{sec:questions}, which features several counterexamples and open questions.  It includes an example of a tree for which the frog model does not have a recurrent regime, thus confirming that for unbounded offspring distributions the almost-sure result in Theorem \ref{theorem:mrecthjg} cannot be extended to \emph{every} Galton-Watson tree.  On top of this, we also construct an example of a rooted tree where each vertex has at least $2$ children for which the frog model does not have a {\BLUE $0$-$1$} law (i.e.\ there is a non-trivial intermediate phase between its recurrent and transient regimes).  These two examples show that the content of Theorem \ref{theorem:mrecthjg} may not be upgraded from an almost-sure statement to a sure statement.

	{\RED \subsection{Notation}
		
		The proof uses many couplings between various altered versions of the frog model in addition to many standard objects from probability on Galton-Watson trees.  For convenience, we review notation here to create a centralized location for the reader to turn back to for reference.
		
		\subsubsection*{Probability on Random Trees}
		Throughout, we will utilize various notation from probability on trees, which we briefly review here.  For a rooted tree $T$ and a vertex $v$, the \emph{level} of $v$ denoted $|v|$ is the distance from $v$ to the root.  We write $T_n$ for the set of level $n$ vertices in {\BLUE $T$}.  The descendant subtree $T(v)$ is the induced subtree consisting of all vertices $w$ for which the shortest path from $w$ to the root passes through $v$.  For a non-root vertex $v$, let $\Par{v}$ denote its \emph{parent}, i.e.\ the neighbor of $v$ lying on the shortest path from $v$ to the root.
		
		Given a rooted tree $T$, let $\partial T$ denote the set of infinite non-backtracking paths in $T$ starting at the root.  If a random walk on $T$ is almost-surely transient then {\BLUE we} say that $T$ is \emph{transient}.  On a transient tree $T$, we define \emph{loop-erased random walk} to be the trajectory of a random walk on $T$ with all loops removed.  The trajectory of a loop-erased random walk on $T$ starting at the root is thus an element of $\partial T$, and so we may define the \emph{harmonic measure} $\HARM:=\HARM_T$ to be the law of this random element of $\partial T$.  For a vertex $v \in T$, $\HARM(v) := \HARM(\{\gamma \in \partial T : v \in \gamma \} )$ denotes the probability that a loop-erased random walk escapes through $v$.  Since supercritical Galton-Watson trees are almost-surely transient (see, e.g., \cite{LP}), we will {\BLUE frequently} make use of the harmonic measure.
		
		For a random variable $Z \in \{0,1,\ldots \}$ with $\E[Z] > 1$, let $\GW$ denote the Galton-Watson measure on rooted trees obtained by conditioning the Galton-Watson process with offspring distribution $Z$ on non-extinction.  Similarly, let $\AGW$ denote the \emph{augmented} Galton-Watson measure, i.e.\ the law of a Galton-Watson tree where we attach an additional Galton-Watson tree to the root and condition on non-extinction.  For each $n \in \mathbb{N}$, we will also consider the measure $\AGW_n$, which is {\BLUE obtained} by sampling from $\AGW$ and {\BLUE then} shifting to a vertex on {\RED level} $n$ sampled via $\HARM$.  In other words, we may sample from $\AGW_n$ by sampling a tree $T$ according to $\AGW$, and performing a loop-erased random walk started at the root; if $v_0,v_1,\ldots$ denotes the trajectory of the loop-erased walk, then we take $T(v_n)$ to be our sample of $\AGW_n$.

		Throughout, $\mathbf{T}$ will denote a random tree which will be taken from $\GW$,  $\AGW$ or $\AGW_n$ which will be explicitly described in context.  For a measure $\mu$---e.g.\ $\GW$ or $\AGW$---{\BLUE we} write $\E_\mu[\cdot]$ to denote the expectation with respect to $\mu$.  
		
		\subsubsection*{Frog Models}
		
		Many altered versions of the frog model will be used in the proof.  A common thread for all models is that the trajectories of particles---once activated---are mutually independent; additionally, the distribution of sleeping frogs is always $\Poiss(\lambda)$.
		
		For a rooted tree $T$ and parameter $\lambda \geq 0$, let $\FM_T^{(\lambda)}$ denote the law of the frog-model on $T$ with i.i.d.\ $\Poiss(\lambda)$ sleeping frogs at each non-root vertex.  In Section \ref{sec:0-1}, we will also consider the frog model $\FM_T^{(\lambda+)}$ where there are $\Poiss(\lambda)$ sleeping frogs at each vertex, including the root.  
		
		The proof of recurrence requires multiple altered versions of the frog model.  The \emph{truncated frog model} defined in Section \ref{sec:truncated} is denoted by $\TFM_T^{(\lambda)}$.  For the purposes of self-similarity, in Section \ref{sec:jgsm3} we also introduce the \emph{augmented truncated frog model} add an extra leaf to the root of $T$ to obtain the tree $T^+$ and alter the dynamics of the walk at the root; this altered model is denoted $\ATFM_{T}^{(\lambda)}$.
		
		The results throughout are \emph{quenched} results, meaning that we consider a random tree ${\bf T}$ and prove statements about various frog models for almost-every ${\bf T}$.  As a result, we will often write $\FM_{\bf T}^{(\lambda)}$ for instance, which denotes the frog model on the random tree ${\bf T}$.  Every time the random tree ${\bf T}$ appears, the corresponding probability measure used to generate ${\bf T}$ will be specified; conversely, we refer to arbitrary trees as $T$.  The difference between the two is not of crucial importance, although we maintain this convention throughout for consistency.
		
	}
	
	\section{A {\BLUE $0$-$1$} Law}\label{sec:0-1}
	
	Our primary goal in this section is to prove a {\BLUE $0$-$1$} law for the frog model on Galton-Watson trees.  Here, we let $\FM^{(\lambda)}_T$ denote the probability measure induced by the frog model on a tree $T$ where $\Poiss(\lambda)$ sleeping frogs are placed at each non-root vertex.  For an offspring distribution $Z$ with $\E[Z] > 1$ and $\P(Z = 0) = 0$ let $\GW$ denote the corresponding Galton-Watson measure on rooted trees and let ${\bf T}$ represent a random rooted tree selected according to $\GW$.  Here we write $\mathrm{recurrent}$ for the event that infinitely many particles visit the root.  
	
	\begin{theorem}\label{th:0-1-GW}
		For any $Z$ with $\E[Z] > 1$ and $\P(Z = 0) = 0$, if $\lambda > 0$ satisfies \\ \noindent  $\E_\GW[\FM^{(\lambda)}_{{\bf T}}(\mathrm{recurrent})  ] > 0$, then $\E_\GW[\FM^{(\lambda)}_{{\bf T}}(\mathrm{recurrent})  ] =1\,.$
	\end{theorem}
	
	\medskip
	
	{\RED We note that the only property of the distribution of sleeping frogs used in the proof of Theorem \ref{th:0-1-GW} is that they are i.i.d.\ and take the value of $0$ with positive probability.}  The proof of Theorem \ref{th:0-1-GW} is broken into two main parts: first, an ergodic theory argument proves the statement for \emph{augmented Galton-Watson trees}, i.e.\ a Galton-Watson tree where we attach an additional Galton-Watson tree to the root; second, we show that working with augmented Galton-Watson trees is sufficient to establish the result for Galton-Watson trees.  Throughout this section, we assume that $Z$ satisfies $\E[Z] > 1$ and $\P(Z = 0) = 0$; to reduce clutter, we do not restate this assumption for each lemma in this section.
	
	\subsection{The Proof for Augmented Galton-Watson Trees}
	
	The ergodic theory argument that we use is heavily indebted to the groundbreaking work \cite{LPP} and the altered versions that appear in \cite{LPP-biased} and \cite[Chapter 17]{LP}; following their lead, let $\AGW$ denote the augmented Galton-Watson measure on rooted trees.  The purpose of adding an extra child to the root is that now the root---on average---looks the same as any other vertex, thereby making the problem more amenable to ergodic theory arguments; to increase the self-similarity, let $\FM^{(\lambda +)}_T$ denote the measure induced by the frog model on $T$ where $\Poiss(\lambda)$ sleeping frogs are placed at each vertex \emph{including} the root (note that we can think of the $\Poiss(\lambda)$ sleeping frogs added to the root as being immediately activated by the single active frog positioned at the root).  We will work on a large measure space containing all of the information necessary for the frog model: define $\TPPT$ to be the set of rooted trees decorated with an infinite path coming from the root, a non-negative {\BLUE integer} $n_v$ associated to each vertex, and $n_v$ paths starting from each vertex $v$ {\BLUE (for $v$ equal to the root, this is in addition to the infinite path already referenced)}.  Define $\AGW \times \SRW \times \Poiss_\lambda \times \SRWS$ on $\TPPT$ to be the measure where the measure on trees is $\AGW$, the infinite path from the root is assigned the law of an independent simple random walk, the numbers $n_v$ are i.i.d. $\Poiss(\lambda)$, and the laws of the $n_v$ paths are mutually independent simple random walks starting at $v$.  Note that if we place $n_v$ sleeping frogs at each vertex, and use the assigned paths to be their trajectories---should they awaken---and use the path at the root to be the trajectory of the first awake frog, then this measure space can be used for $\FM^{(\lambda +)}_{{\bf T}}$.  
	
	We will decompose $\TPPT$ into the space of trees and paths---which we denote $\TP$---and think of the particles and their trajectories as decorating it; this will allow us to lean on the work of \cite{LPP}.  For a given $\omega \in \TPPT$, define the shift operator $S$ as follows: let $v$ be the first vertex (after the root) along the path component of $\omega$ that is assigned to the root: \begin{itemize}
		\item the tree of $S(\omega)$ is the tree of $\omega$ with root shifted to be $v$.
		\item the path from the root $(x_0,v,x_2,\ldots)$ is changed to $(v,x_2,\ldots)$.
		\item the numbers $n_v$ and other trajectories are unchanged.
	\end{itemize} 
	
	\noindent
	With these definitions in place, we note that 
	
	\begin{lemma}  The system $$(\TPPT,\AGW\times \SRW \times \Poiss_\lambda \times \SRWS, S)$$ is stationary.
	\end{lemma}
	\begin{proof}
		If we project to the space $\TP$ that ignores the numbers $n_v$ and associated paths, then the stationarity of $(\TP,\AGW \times \SRW, S)$ is proven in \cite[Theorem 17.11]{LP}; since we have merely decorated the space with independent variables that do not depend on the location of the root, stationarity in our bigger space follows immediately.
	\end{proof}
	
	In addition to the above lemma, \cite{LPP-biased} proves that we in fact can $\AGW \times \SRW$-almost-surely decompose each tree-path pair as a collection of i.i.d.\ \emph{slabs}, which build the tree and path together simultaneously in blocks; {\RED for more details concerning this decomposition, see \cite[Section 17.3]{LP}.  While we do not need a precise definition of slabs, we define them informally: we say that a random walk on a tree \emph{regenerates} when it crosses an edge for the first and last time simultaneously.  The portion of the tree between regeneration points, along with the path of the walk through this portion, is called a slab.  As noted, the slabs are in fact i.i.d.\ when the tree is sampled from $\AGW$ and the walk follows $\SRW$ \cite[Section 17.3]{LP}.}  In particular, this means that for every fixed $n$, there is an almost-surely finite $\tau$ so that the first $n$ {\RED levels} of $S^\tau(\omega)$ as well as the path until exiting this tree are independent of the first $n$ {\RED levels} of $\omega$.  We will use this to show that the above system is in fact ergodic.
	
	Let $\F$ be the $\sigma$-field on which the measure $\mu:= \AGW\times \SRW\times \Poiss_\lambda \times \SRWS$ is defined.  For each natural number $n$, let $\F_{n}$ denote the $\sigma$-field induced by: \begin{itemize}
		\item the first $n$ {\RED levels} of the tree
		\item the path from the root until first exiting the first $n$ {\RED levels}
		\item the particle configuration for the first $n$ {\RED levels}
		\item the trajectories of these particles until first exiting the first $n$ {\RED levels}.
	\end{itemize} 
	
	\noindent
	Since almost-surely all of these paths exit the first $n$ {levels} in finite time, there are only countably many possible configurations for each $n$.  Define $\mathcal{U} := \bigcup_{n} \F_n$ and note that the smallest $\sigma$-field containing $\mathcal{U}$ is $\mathcal{F}$.  
	
	\begin{lemma}
		The system $$(\TPPT,\mu, S)$$ is ergodic.
	\end{lemma}
	\begin{proof}
		In order to show ergodicity, we will prove that the system is \emph{(strong) mixing}, i.e.\ that for each $A, B \in \F$ we have \begin{equation} \label{eq:mixing}
			\lim_{n\to\infty}\mu[A \cap S^{-n} B] \to \mu[A]\mu[B]\,.
		\end{equation}To do this, we will first show \eqref{eq:mixing} for events in $\mathcal{U}$.  Let $A, B \in \mathcal{U}$ and let $n$ be large enough so that both $A, B \in \F_n$.  Since we may find an almost-surely finite stopping time $\tau$ so that shifting $S^{\tau}$ moves sufficiently many slabs away from the root so that the first $n$ {\RED levels} of the system are independent of the first $n$ {\RED levels} of the system before shifting, \eqref{eq:mixing} follows for such $A,B$ from the fact that $\tau < \infty$ a.s. 
		
		To show mixing for all events in $\F$, we use Dynkin's $\pi$-$\lambda$ Theorem.  Define $$\mathcal{V}:= \{C \in \F : \mu[A \cap S^{-n} C] \xrightarrow{n \to \infty} \mu[A]\mu[C] \text{ for all }A \in \mathcal{U}  \}\,.$$Note that $\mathcal{U}$ is a $\pi$-system and $\mathcal{V}$ is a $\lambda$-system.  Further, we have shown that $\mathcal{U}\subset \mathcal{V}$; hence, by Dynkin's $\pi$-$\lambda$ Theorem, $\mathcal{F}\subset \mathcal{V}$, implying $\mathcal{F} = \mathcal{V}$.  Now, define 
		
		$$\mathcal{W}:= \{D \in \F : \mu[D \cap S^{-n} C] \xrightarrow{n \to \infty} \mu[D]\mu[C] \text{ for all }C \in \mathcal{F}  \}\,.$$This collection $\mathcal{W}$ is again a $\lambda$-system and---by the previous $\pi$-$\lambda$ argument---contains $\mathcal{U}$.  Therefore we again have $\mathcal{F}\subset \mathcal{W}$ implying \eqref{eq:mixing} holds for all $A,B \in \F$.
	\end{proof}
	
	From here, establishing a {$0$-$1$ law for recurrence} of $\FM^{(\lambda)}$ on augmented Galton-Watson trees follows from the ergodic theorem.

	\begin{lemma}\label{lem:0-1-AGW}
		If $\E_{\AGW}[\FM^{(\lambda)}_{{\bf T}}(\mathrm{recurrent})] > 0$,  then $\E_\AGW[\FM^{(\lambda)}_{{\bf T}}(\mathrm{recurrent})] = 1$.
	\end{lemma}
	\begin{proof}
		By ergodicity, we have that 
		$$\frac{1}{n}\sum_{k = 0}^{n-1} \one\{ \mathrm{recurrent}\}(S^k (\omega)) \xrightarrow{n\to\infty} \E_{\AGW}[\FM^{(\lambda +)}_{{\bf T}}(\mathrm{recurrent})]$$
		almost surely.  Since $\FM^{(\lambda +)}$ dominates the frog model, this limit is positive.  In particular, this means that we can shift so that we yield a configuration that is recurrent for $\FM^{(\lambda +)}$.  Since shifting may only reduce the set of particles that awaken---and the trajectories are unchanged aside from the initial particle---this implies that, for ${\sf AGW}$ almost surely every $T$, infinitely many particles visit some (not necessarily fixed) vertex of $T$ with probability $1$.  Now assume that for such a tree $T$, we have ${\sf FM}^{(\lambda +)}_T(\text{recurrence})<1$.  This would then have to imply that there exists some \emph{fixed} non-root vertex $v$ in $T$ for which ${\sf FM}^{(\lambda +)}_T(v\ \text{is hit i.o. but root is not})>0$,
		{\RED where we write i.o.\ for \emph{infinitely often}}
		(since the probability of any single frog hitting the root infinitely often is $0$ {due to the almost-sure transience of random walks on supercritical Galton-Watson trees}, we can take the term `i.o.' to mean being hit by infinitely many \emph{distinct} frogs).  
		
		Now let $N$ be some positive integer and, for each $n\geq 0$, let $A_n$ represent the event that at least $n+1$ distinct frogs hit $v$ after time $N$ and, among the first $n$ of these, none go on to hit the root afterwards.  If, in addition, we let $p(v_1,v_2)$ denote the probability that simple random walk begun at $v_1$ ever hits $v_2$, then from here we observe that for each $n$ we have 
		$${\sf FM}^{(\lambda +)}_T(A_n)={\sf FM}^{(\lambda +)}_T(A_0)\prod_{i=1}^n {\sf FM}^{(\lambda +)}_T(A_i|A_{i-1})\leq\big(1-p(v,{\bf 0})\big)^n.$$
		Noting that this last expression goes to $0$ as $n\to\infty$, and then allowing $N$ to go to infinity, we see that we cannot in fact have ${\sf FM}^{(\lambda +)}_T(v\ \text{is hit i.o. but root is not})>0$.  This contradiction then establishes that if $\E_{\AGW}[\FM^{(\lambda)}_{{\bf T}}(\mathrm{recurrent})] > 0$, then it must follow that $\E_{\AGW}[\FM^{(\lambda +)}_{{\bf T}}(\mathrm{recurrent})]=1$.  Conditioning on the event that the number of sleeping frogs placed at the root is $0$ completes the proof.
	\end{proof}
	
	\subsection{Connecting $\AGW$ to $\GW$}
	
	In order to prove Theorem \ref{th:0-1-GW} using Lemma \ref{lem:0-1-AGW}, we will need to show two implications: Namely, that $\E_\GW[\FM^{(\lambda)}_{{\bf T}}(\mathrm{recurrent})] > 0$ implies $\E_\AGW[\FM^{(\lambda)}_{{\bf T}}(\mathrm{recurrent})] > 0$, and that $\E_\AGW[\FM^{(\lambda)}_{{\bf T}}(\mathrm{recurrent})] = 1$ implies \\ \noindent $\E_\GW[\FM^{(\lambda)}_{{\bf T}}(\mathrm{recurrent})] = 1$.  We begin with the former:  
	
	\begin{lemma}\label{lem:AGW-to-GW}
		If $\E_\GW[\FM^{(\lambda)}_{{\bf T}}(\mathrm{recurrent})] > 0$ then $\E_\AGW[\FM^{(\lambda)}_{{\bf T}}(\mathrm{recurrent})] > 0\,.$
	\end{lemma}
	
	\begin{proof}
		We will consider a model that is dominated by $\AGW \times \FM^{(\lambda)}$.  To start, generate a copy of $Z+1$ and call its value $k$.  Generate $k$-many Galton-Watson trees $T^{(1)},\ldots, T^{(k)}$ and place i.i.d.~$\Poiss(\lambda)$ inactive particles at each non-root vertex.  Label the roots of these trees $v_1,\ldots, v_k$ and connect the $Z+1$-many roots to another vertex $\rtt$, which will be taken to be the root of this larger tree; place a single active particle at $\rtt$.  {\RED Since each $v_j$ was the root of a $T^{(j)}$, note that there are no sleeping frogs at any of the $v_j$.} The broad idea is that we will break the tree up into $k+1$ pieces: the set $S:= \{\rtt, v_1,\ldots,v_k\}$ together with $T^{(j)} \setminus \{v_j\}$ for each $j \in \{1,\ldots,k\}$.   Only one of these $k+1$ sets will have particles moving at any given time.  The frog model rule that inactive particles are activated when touched by active particles will still be in effect, however, since no two distinct pieces of the $k+1$ parts that we've separated $T$ into are permitted to have particles in motion simultaneously, the designation ``active" no longer implies a particle is necessarily in the process of moving. 
		
		At time $t=0$ the active particle initially positioned at $\rtt \in S$ begins performing a simple random walk, continuing until it moves into one of the sets $T^{(j)} \setminus \{v_j\}$.  Upon entering this set, this particle continues its random walk, activating sleeping particles along the way, which in turn perform simple random walks activating the sleeping particles that they encounter, and so on (i.e. the normal frog model dynamics apply inside of $T^{(j)} \setminus \{v_j\}$).  This persists until a particle arrives at $v_j$ (which may never happen).  
		When one of these active particles arrives at $v_j$, all other active particles in $T^{(j)} \setminus \{v_j\}$ pause their walks
		and the particle that hit $v_j$ performs its simple random walk in $S$ until exiting, i.e. until entering into another $T^{(i)} \setminus \{v_i\}$; if two particles arrive simultaneously, we break ties arbitrarily.  Now active particles in $T^{(i)} \setminus \{v_i\}$ evolve until a particle arrives at $v_i$, and so on (note that if $i=j$, then all particles in $T^{(j)} \setminus \{v_j\}$ that are already active have their walks resume upon entry of this particle into $T^{(j)} \setminus \{v_j\}$).
		
		The key feature of this model is that, when looking at a single one of the trees $T^{(j)}$, it is simply the frog model stopped---and possibly later restarted---when a particle hits the root of $T^{(j)}$.  This is because whenever a particle exits $T^{(j)}$, evolution inside $T^{(j)}$ stops until a particle enters $T^{(j)} \setminus \{v_j\}$.  Note further that, due to the time independence property of the frog model, the model we have described is simply the frog model on augmented Galton-Watson trees where we possibly ignore the trajectories of many frogs (including those initially positioned at the children of the root). 
		In particular, the probability of recurrence for this model is a lower bound for the
		probability of recurrence of $\AGW\times \FM^{(\lambda)}$.
		The event of recurrence 
		must occur provided on each $T^{(j)}$ 
		the frog model there---i.e.\ the particles and trajectories assigned to them in the larger model---is recurrent.  Since by assumption each $T^{(j)}$ has a positive probability of this occurring and these events are independent for different $j$, we have that there is a positive probability of recurrence and non-extinction of this model, and thus of $\AGW\times \FM^{(\lambda)}$.
	\end{proof}
	
	\medskip
	We now establish the second implication needed to complete the proof of Theorem \ref{th:0-1-GW}.
	
	\begin{lemma}\label{agwasrigwasr}
		If  $\E_\AGW[\FM^{(\lambda)}_{{\bf T}}(\mathrm{recurrent})]=1$, then $\E_\GW[\FM^{(\lambda)}_{{\bf T}}(\mathrm{recurrent})]=1$.
	\end{lemma}
	
	\begin{proof}
		The proof is by a so-called ``local modification'' argument.  
		If $T$ is a tree for which $\FM^{(\lambda)}_T(\mathrm{recurrent})<1$, then there must be a finite set of non-root vertices $v_1,\dots,v_j\in T$ such that, with positive probability, no particles from outside the set $v_1,\dots,v_j\in T$ ever {\RED visit} the root.  Since there is positive probability that no sleeping frogs reside at any of the vertices $v_1,\dots,v_j\in T$, this then means that $\FM^{(\lambda)}_T(\text{no {\RED initially sleeping} particles {\RED visit the} root})>0$, which implies that there exists at least one vertex $v\in T_1$ such that
		$$\FM^{(\lambda)}_T(\text{no particles from}\ T(v)\ \text{hit root}|\text{particle starting at root hits}\ v\ \text{on 1st step})>0.$$  
		Now let $D$ denote the set of all $T$ that have such a vertex $v\in T_1$.  Since ${\sf AGW}$ only differs from ${\sf GW}$ on account of the root being assigned an extra child (which is itself the root of a subtree with distribution ${\sf GW}$), this then implies that ${\sf AGW}(D)\geq{\sf GW}(D)$.  Since we've established that a tree $T$ satisfies $\FM^{(\lambda)}_T(\mathrm{recurrent})<1$ if and only if $T\in D$, we can then conclude that $\AGW\Big(\FM^{(\lambda)}_{{\bf T}}(\mathrm{recurrent})<1\Big)\geq\GW\Big(\FM^{(\lambda)}_{{\bf T}}(\mathrm{recurrent})<1\Big)$, which completes the proof of the lemma.
	\end{proof}

	\begin{proof}[Proof of Theorem \ref{th:0-1-GW}] 
		Let $\lambda > 0$ so that $\E_\GW[\FM^{(\lambda)}_{{\bf T}}(\mathrm{recurrent})] > 0$.  Then by Lemma \ref{lem:AGW-to-GW}, it must be that $\E_\AGW[\FM^{(\lambda)}_{{\bf T}}(\mathrm{recurrent})] > 0$.  Applying Lemma \ref{lem:0-1-AGW} then shows $\E_\AGW[\FM^{(\lambda)}_{{\bf T}}(\mathrm{recurrent})] = 1$.  Lemma \ref{agwasrigwasr} completes the proof.
	\end{proof}

	\section{Transience}
	\label{sec:transience}
	In this brief section, we establish a basic transience result that applies for all rooted trees without leaves or pipes, {\RED i.e.\ all rooted trees whose non-root vertices have degree at least $3$}.  Specifically, we obtain a lower bound on the value $\lambda_1(T):=\text{sup}\{\lambda:{\sf FM}^{(\lambda)}_T(\text{transience})=1\}$  with respect to the minimum degree of $T$, which is the direct analogue of the transience result achieved by Hoffman, Johnson, and Junge in \cite{HJJ1} for regular trees.
	
	\medskip
	\begin{remark}
		Note that the reason we refer to the quantity $\lambda_1(T)$ here (rather than $\lambda_c(T)$) is because we are working to achieve a result that applies for \emph{every} rooted tree without leaves or pipes, rather than just almost surely every tree generated by some offspring distribution.  Hence, we cannot assume that the {\BLUE $0$-$1$} law obtained in the previous section necessarily holds.  Indeed, a counterexample is presented in Lemma \ref{lem:nontrivial-intermediate}.
	\end{remark}
	
	\noindent
	\subsection{Using minimal degree to bound $\lambda_1$}\label{sec:minimal-degree} We now present a result that relates $\lambda_1(T)$, for a tree $T$ without leaves or pipes, to the minimum degree for vertices in $T$.  While we are largely interested in the critical value of the Poisson mean $\lambda_1(T)$, the proof applies to any nonnegative integer valued random variable with the specified mean, and thus we state the theorem in that generality.  The statement, as well as the proof, mirrors Proposition 15 from \cite{HJJ1}, which consists of the analogous result for $n$-ary trees; {\RED we include the details here for the sake of obtaining a good lower bound on $\lambda_c$ in Theorem \ref{theorem:mrecthjg}:}
	
	\begin{theorem} \label{th:transiencemd}
		Let $T$ be a rooted tree for which all vertices have at least $k\geq 2$ children.  Then the frog model on $T$ with i.i.d. $\eta$ frogs per non-root vertex is transient provided ${\bf E}[\eta]<\frac{(k-1)^2}{4k}$.
	\end{theorem}
	
	\begin{proof} 
		We begin by defining the branching random walk model on $T$ that starts with a single particle positioned at the root at time $0$, and where particles perform independent simple random walks, each one giving birth to $\eta$ additional particles every time it takes a step away from the root.  Letting $Y$ represent the total number of returns to the root for this model, we note that since $Y$ stochastically dominates $V$ (the number of returns to the root for the frog model with $\eta$ frogs per non-root vertex), it will suffice to establish the desired result for the branching model.  Adopting the notation from the proof of Proposition 15 in \cite{HJJ1}, we let $F_n$ represent the set of active particles at time $n$, and for every particle $f\in F_n$, we denote its distance from the root as $|f|$.  Next we define the weight function $$W_n=\sum_{f\in F_n}\alpha^{|f|},$$where $\alpha=\big(({\bf E}[\eta]+1)k\big)^{-1/2}$.  Now for any frog $f$ positioned at time $n$ at a non-root vertex with $j\geq k$ children, the expected contribution that $f$, along with all of its progeny that are born at time $n+1$, makes to $W_{n+1}$ is equal to\begin{align}\label{cewnp1nr}\frac{1}{j+1}\alpha^{|f|-1}+\frac{j}{j+1}\big({\bf E}[\eta]+1\big)\alpha^{|f|+1}&=\alpha^{|f|}\bigg(\frac{1}{j+1}\alpha^{-1}+\frac{j}{j+1}\big({\bf E}[\eta]+1\big)\alpha\bigg)\\&=\alpha^{|f|}\bigg(\frac{\alpha^{-1}}{k+1}\cdot\frac{(k+j)(k+1)}{(j+1)k}\bigg)\nonumber\\&=\alpha^{|f|}\cdot\frac{\alpha^{-1}}{k+1}\bigg(2-\frac{(k-1)(j-k)}{k(j+1)}\bigg)\nonumber\\&\leq\alpha^{|f|}\cdot\frac{2\alpha^{-1}}{k+1}.\nonumber
		\end{align}
		Likewise, in the case where $f$ is at the root, the expected contribution $f$ and its progeny make to $W_{n+1}$ is 
		$$\alpha\big({\bf E}[\eta]+1\big)=\frac{\alpha^{-1}}{k}\leq\alpha^{|f|}\cdot\frac{2\alpha^{-1}}{k+1}.$$
		Hence, combining this with \eqref{cewnp1nr} and summing over all $f\in F_n$ we get 
		\begin{equation}\label{cebwnp1wrtwn}{\bf E}[W_{n+1}|W_n]\leq\frac{2\alpha^{-1}}{k+1}W_n.\end{equation}
		Now defining $m:=\frac{2\alpha^{-1}}{k+1}$ and noting that \eqref{cebwnp1wrtwn} implies that $\frac{W_n}{m^n}$ is a nonnegative supermartingale, we see that $\frac{W_n}{m^n}$ must be almost surely convergent.  Combining this with the fact that ${\bf E}[\eta]<\frac{(k-1)^2}{4k}\iff m<1$, we can now conclude that $W_n\longrightarrow 0$ a.s., thus establishing transience and completing the proof.
	\end{proof}

	\section{Basic tools for random walk and Galton-Watson trees} \label{sec:tools}
	
	Before proceeding to the proof of recurrence, we introduce some basic tools concerning random walk on trees and comparisons between different probability measures on trees.  These will form the toolbox throughout the proof of recurrence.  Many of these statements have technical and somewhat standard proofs, which will be deferred to an appropriate appendix.  
	
	\subsection{Properties of the harmonic measure and simple random walk}  \label{sec:SRW}
	We begin with a basic calculation that is the primary use of our assumption that $Z \geq 2$.
	\begin{lemma}\label{lem:1/2}
		Let $T$ be an infinite tree with root $\bf 0$ so that each vertex is of degree at least $3$.  Then for all $v \in T_1$ the probability that a simple random walk on $T$ starting at $\bf 0$ ever hits $v$ is at most $1/2$.  We thus have  $\HARM_{T}(v) \leq 1/2$.
	\end{lemma}
	\begin{proof}
		Let $X_n$ be the distance from the random walker at time $n$ to $v$.  Consider now a random walk $\{Y_n\}$ with drift on $\bf Z$ started at $1$ where we take right-ward steps with probability $2/3$ and left-ward steps with probability $1/3$.  Up until $Y_n = 0$, we may couple the random walks $(X_n,Y_n)$ so that $X_n$ is larger than $Y_n$.  Computing directly that $\P(Y_n >0 \text{ for all }n) = 1/2$ completes the proof.
	\end{proof}

	We will want to compare the probability a simple random walk visits the root of a tree to its harmonic measure.  As such, for a rooted tree $T$ and a vertex $u \in T$ define $p_0(u)$ be the probability that a loop-erased random walk starting at $u$ ever hits the root.  
	
	\begin{lemma}\label{lem:crphmjg}
		There is a universal constant $C > 0$ so that for any rooted tree $T$ with minimum degree $m\geq 3$ and every $v \in T_2$ and every $u\in T(v)$, we have $$p_0(u)\geq C\cdot\frac{{\sf HARM}_{T(v)}(u)}{|T_1(u)|\cdot|T_1(\Par{v})|}.$$
	\end{lemma}
	
	Finally, we will want to compare the harmonic measure to simple random walk without loop erasure.
	
	\begin{lemma}\label{lem:lljgyhi}
		In a simple random walk on a rooted tree $T$ and a vertex $v \in T_n$, let $f(v)$ be the probability that $v$ is the first level $n$ vertex visited by the random walk.  There exists a universal constant $C\in(1,\infty)$ such that, for any infinite rooted tree $T$ where all vertices have at least two children, and any non-root vertex $v$ of $T$ we have \begin{equation}\label{fboundbh}\frac{1}{C}\cdot{\sf HARM}_T(v)\leq f(v)\leq C\cdot{\sf HARM}_T(v).
		\end{equation}
	\end{lemma}
	
	Both Lemma \ref{lem:crphmjg} and \ref{lem:lljgyhi} are proved in Appendix \ref{app:hitting}.
	
	\subsection{Cutting out subtrees with a large first generation} \label{sec:cutting}
	Since we will use the harmonic measure as a weighing for counting the number of activated vertices, we will show that it is rare for the harmonic measure to have a big contribution from vertices with many children.  
	
	\begin{lemma}\label{lem:base}
		There are universal constants $C,c > 0$ so that for all $N \geq C \mu$ and $n$ we have $${\sf GW}\bigg(\sum_{v'\in{\bf T}_n}{\sf HARM}_{{\bf T}}(v')\cdot\mathbf{1}_{|{\bf T}_1(v')|\geq N}\geq\frac{1}{4}\bigg)\leq \exp\left(-c 2^n\right)\,.$$
	\end{lemma}
	
	\begin{proof}
		
		Let $T$ be a rooted tree for which all vertices have at least two children.  If $v$ is a level $n$ vertex in $T$ and we let $v_0,v_1,\dots,v_n=v$ be the path going from the root to $v$, then since 
		$${\sf HARM}_T(v)=\prod_{j=0}^{n-1}{\sf HARM}_{T(v_j)}(v_{j+1}),$$
		Lemma \ref{lem:1/2} implies that ${\sf HARM}_T(v)\leq 2^{-n}$.  Now letting $u_1,\dots,u_k$ represent an enumeration of the level $n$ vertices in $T$, and letting $f(u_j)$ represent the probability that $u_j$ is the first level $n$ vertex hit by simple random walk on $T$ beginning at the root, we see that the above exponential bound combined with Lemma \ref{lem:lljgyhi} implies that $f(u_j)\leq C 2^{-n}$ for each $j$.
		
		To obtain the statement bound in the Lemma we will bound an exponential moment.  Observe that if we condition on the first $n$ levels of the random tree ${\bf T}$ matching those of $T$ (we write this as ${\bf T}_n\approx T_n$), then we may use conditional independence to find that for any $t > 0$ we have 
		\begin{align*}
			\E_{{\sf GW}}&\left[\exp\left(t\sum_{j = 1}^{k} f(u_j)\mathbf{1}_{|{\bf T}_1(u_j)|\geq N}\right)\,\bigg|\,{\bf T}_n\approx T_n\right] \\
			&\qquad=\prod_{j=1}^k \E_{{\sf GW}}\left[\exp\left(tf(u_j)\mathbf{1}_{|{\bf T}_1(u_j)|\geq N}\right) \, \bigg| \, {\bf T}_n \approx T_n\right] \\
			&\qquad=\prod_{j=1}^k\left( 1+\P(Z \geq N)\big(e^{tf(u_j)}-1\big) \right).
		\end{align*}
		
		In addition, since $e^x\leq 1+2x$ for all $ x\in[0,1]$, it follows from the above calculation that for the choice $t = \frac{1}{C}2^n$ and writing $p_N = \P(Z \geq N)$ we have
		\begin{align*}\E_{{\sf GW}}\Big[\exp\left(t\sum_{j=1}^k f(u_j)\mathbf{1}_{|{\bf T}_1(u_j)|\geq N}\right)\Big|{\bf T}_n\approx T_n\Big]&\leq\prod_{j=1}^k\left( 1+2tp_Nf(u_j) \right) \\
			&\leq\prod_{j=1}^k e^{2tp_N f(u_j)}=e^{2tp_N}.
		\end{align*}
		Markov's inequality that then shows that for any $r>0$ we have 
		$${\sf GW}\bigg(\sum_{j=1}^k f(u_j)\one_{|{\bf T}_1(u_j)|\geq N}\geq r\Big|{\bf T}_n\approx T_n\bigg)\leq e^{-(r-2p_N)\frac{1}{C}2^n}.$$
		Now once again applying Lemma \ref{lem:lljgyhi}, while also noting that the expression on the right in the above inequality does not depend on $T_n$, we can conclude that $${\sf GW}\bigg(\sum_{v'\in{\bf T}_n}{\sf HARM}_{{\bf T}}(v')\mathbf{1}_{|{\bf T}_1(v')|\geq N}\geq Cr\bigg)\leq e^{-(r-2p_N)\frac{1}{C}2^n}.$$
		Finally, setting $r=\frac{1}{4C}$ and noting that for $N \geq C \mu$ we have $2p_n \leq r/2$ completes the proof.
	\end{proof}

	In the proof of recurrence we will often require that many levels of the tree satisfy the event in Lemma \ref{lem:base}.  To this end, we introduce the following definitions.
	
	\begin{defn}[Bad tree events]
		For a given $N,m$ and $i$ define the collection of rooted trees $A_i$ by $$A_i = \left\{\sum_{v'\in{T}_i}{\sf HARM}_{{ T}}(v')\mathbf{1}_{|{ T}_1(v')|\geq N}\geq\frac{1}{4} \right\}.$$
		Define $$A = \bigcup_{\frac{m}{2}<i\leq m} A_i\,.$$
	\end{defn}

	\begin{lemma}\label{lem:A-rare}
		There are universal constants $C,c > 0$ so that for all $N \geq C\mu$ and all $m$ we have $$ \GW(A) \leq C e^{-c 2^{m/2}}\,.$$
	\end{lemma}
	\begin{proof}
		Apply Lemma \ref{lem:base} along with a union bound over $i$.  Adjusting the constant $c$ completes the Lemma.
	\end{proof}

	\subsection{Comparisons between measures on trees} \label{sec:comp}
	
	We will need two families of measures on trees that are slightly different from $\GW$ and $\AGW$ but are closely related.  First, sample $\bT$ according to $\AGW$, and let $v_0,v_1,\ldots$ be the vertices of an infinite nonbacktracking path from the root that is sampled according to the harmonic measure ${\sf HARM}_{{\bf T}}$.  Define $\AGW_n$ denote the law of the rooted tree $\bT(v_n)$.  Note that since the vertex $v_n$ is chosen according to the harmonic measure, this has the effect of biasing ${\bf T}(v_n)$ to be larger than $\mathbf{T}$ in some sense.  As such, $\AGW_n$ is distinct from $\GW$, however we will show that the two measures are quite close.

	\begin{lemma}\label{lem:lljgyhiyy}
		There exists a universal constant $C\in(1,\infty)$ such that, if $Z$ is an offspring distribution satisfying $\Prob(Z\geq 2)=1$, then $$\frac{1}{C}\leq\frac{d{\sf AGW}_n}{d{\sf GW}}\leq C\ \ \ {\sf GW}-\text{a.s.}\ \forall\ n\geq 1.$$
	\end{lemma}
	
	We will also need a law on random trees that have a given first generation and then shift by the harmonic measure.

	\begin{lemma}\label{lem:GW1}
		For each $j \in \mathbf{N}$ define the measure $\GW_1^{(j)}$ as follows: let $\mathbf{T}$ be sampled from $\GW$ conditioned on $|\mathbf{T}_1| = j$.  Choose $v \in \mathbf{T}_1$ according to the harmonic measure on $\mathbf{T}$.  Then set $\GW_1^{(j)}$ to be the law of $\mathbf{T}(v)$.   Then $$\frac{1}{2} \leq \frac{d \GW_1^{(j)}}{d\GW} \leq 2\,.$$
	\end{lemma}
	
	Lemmas \ref{lem:lljgyhiyy} and \ref{lem:GW1} are proved in Appendix \ref{app:comp}.
	
	\section{Recurrence}  \label{sec:proofmain}
	The purpose of this Section is to show the existence of a recurrent regime for Galton-Watson trees, which we state explicitly as Theorem \ref{theorem:mrecthjg}.  We begin by introducing two variants of the frog model: the truncated frog model and the augmented truncated frog model; the purpose of the truncated frog model is for a direct comparison to the frog model (Lemma \ref{lem:truncated-frog}), while the purpose of the augmented truncated frog model is to understand the behavior of the truncated frog model when shifted to a vertex $v$ (Lemma \ref{lem:augmented}).  The main technical engine in our proof is Proposition \ref{pr:iteratejg}, which we prove in Section \ref{sec:proof-prop}.
	

	\subsection{Two variants: The truncated frog model and its augmented version} \label{sec:truncated}
	
	In this section we define two variants on the frog model.  The first is the truncated frog model, which we will directly compare the frog model to.  The dynamics of the \emph{truncated frog model} are as follows:
	
	\begin{defn}[Truncated Frog Model Dynamics]  Given a rooted tree $T$ and a parameter $\lambda \geq 0$ the truncated frog model is defined as follows.
		\begin{enumerate}
			\item Like the ordinary frog model, this model begins with a single active particle at the root, and i.i.d.\ $\text{Poiss}(\lambda)$ inactive particles at all non-root vertices.
			
			\item An inactive particle is activated when the vertex at which it resides is landed on by an active particle.  Upon activation, particles perform independent loop-erased random walks, which terminate upon hitting the root.
			
			\item In addition, any time an active particle takes a step away from the root and lands on a vertex which has already been landed on by at least one other active particle, the particle is eliminated.  If more than one particle simultaneously land on a vertex which had not previously been landed on by an active particle, all but one of these particles are eliminated.
		\end{enumerate}
	\end{defn}

	\noindent

	We first note that there is a natural coupling between the frog model and the truncated frog model so that the number of particles that visit the root in the frog model is at least as large as the number in the truncated frog model.  
	
	\begin{lemma}\label{lem:truncated-frog}
		Suppose $\lambda$ is so that $\GW$-a.s.\ the truncated frog model with $\Poiss(\lambda)$ sleeping frogs at each non-vertex root is recurrent.  Then the same holds for the frog model $\GW$-a.s.
	\end{lemma}
	\begin{proof}
		We  provide a coupling between the frog model and the truncated frog model.
		Let $T^{(1)}, T^{(2)}$ be copies of a tree $T$ sampled from $\GW$.  On $T^{(1)}$ we place an active particle at the root, we position i.i.d. $\text{Poiss}(\lambda)$ sleeping particles at all non-root vertices, and then run the ordinary frog model.  Now on $T^{(2)}$, we let each non-root vertex begin with the same number of sleeping particles as the corresponding vertex in $T^{(1)}$, and assign to each particle in $T^{(2)}$ a partner in $T^{(1)}$ originating at the same vertex.  We now define a copy of the truncated model on $T^{(2)}$ by having each particle, if activated, proceed along the path obtained by eliminating all loops from the path taken by its partner in $T^{(1)}$ (activating all sleeping particles it encounters along the way), until the particle in $T^{(2)}$ either hits the root, or travels from a parent vertex to a child that has already been landed on by another particle, at which point it is eliminated.  Letting $X_j$ represent the total number of particles that visit the root in the model defined on $T^{(j)}$, we see that because the trajectory of each activated particle in $T^{(2)}$ is a subset of the trajectory of its partner in $T^{(1)}$, and because all activated particles in $T^{(2)}$ have activated partners in $T^{(1)}$, this implies that $X_1\geq X_2$.  In particular, if $X_2 = +\infty$ then $X_1 = +\infty$. 
	\end{proof}

	For any $T,\lambda$ combination (where $T$ is a rooted tree without leaves or pipes) we denote the law of the truncated frog model on $T$ with $\text{Poiss}(\lambda)$ sleeping frogs per non-root vertex as ${\sf TFM}^{(\lambda)}_T$.  The elements comprising the space on which this measure is defined, denoted as ${\sf PathsParticlesTrajectories}_T$, will consist of the following information: A single non-backtracking trajectory starting at the root, a non-negative integer $n_v$ for each non-root vertex $v$ that refers to the number of sleeping frogs initially located there, $n_v$ non-backtracking paths for each non-root vertex $v$, and finally an element of $[0,1]^{\mathbb{N}}$ associated with each non-backtracking path that allows us to break ties (this is needed on account of the last of the three conditions used to define the truncated frog model in the prior subsection).  

	One difficulty of working with the truncated frog model is that the root is distinguished from other vertices.  In particular, if we consider the truncated frog model on $T$ and some vertex $v$, we will want a model to understand how the model looks on the subtree $T(v)$ conditioned on $v$ being activated.  As such, we will primarily work with a slightly altered version of the frog model that we call the \emph{augmented truncated frog model} defined as follows.   
	
	\begin{defn}[Augmented Truncated Frog Model Dynamics] Given a rooted tree $T$, $\lambda \geq 0$, attach  a leaf vertex $v_{\ell}$ to the root of $T$ in order to generate the tree $T^+$.
		\begin{enumerate}
			\item Place i.i.d.\ $\Poiss(\lambda)$ inactive particles at \emph{all} vertices aside from $v_{\ell}$ and additionally place a single active particle at the root.  
			\item An inactive particle is activated when the vertex at which it resides is landed on by an active particle.  Upon activation, particles perform independent loop-erased random walks.  Each time such a walk hits $v_{\ell}$, it terminates.  The original particle placed at the root performs an ordinary loop-erased random walk on $T$, and so never steps to $v_\ell$.
			
			\item As is the case in the truncated frog model, any time an active particle stepping away from the root lands on a (non-leaf) vertex that has already been landed on by another active particle, this particle dies.  If multiple active particles land on a previously unvisited vertex simultaneously, then all but one (chosen uniformly at random) die.
		\end{enumerate}
		
	\end{defn}

	The intuition behind the augmented truncated frog model is that for a non-root vertex $v$, once $v$ is activated there is some number of inactive particles at $v$; this accounts for the additional particles placed at the root.  Additionally, loop erased random walk starting within $T(v)$ in the larger tree $v$ will either escape through $\Par{v}$ or continue downward in $T(v)$; the purpose of adding the leaf $v_\ell$ is to account for the particles that escape through $\Par{v}$. 
	
	The law induced by this model will be denoted as $\ATFM^{(\lambda)}_{T}$, and the space on which it is defined, which includes all of the information associated with ${\sf PathsParticlesTrajectories}_T$, on top of the information pertaining to the additional $n_v$ active particles starting at the root, will be denoted as ${\sf PathsParticlesTrajectories}^*_T$.  As the $n_v$ particles placed at the root are performing loop-erased random walk, each one either escapes downward in the tree $T(v)$ or is terminated at $v_\ell$.  
	
	The relationship between these two models is summarized in the following basic fact.  
	\begin{lemma}\label{lem:augmented}
		For any rooted tree $T$ so that each vertex has at least two children, $v \in T_1$ and $\lambda > 0$ we have $$\TFM_T^{(\lambda)}(v \text{ is activated}\,|\,\Par{v}\text{ is activated}) \geq  \ATFM_{T(\Par{v})}^{(\lambda)}(v \text{ is activated})\,.$$
	\end{lemma}
	\begin{proof}
		Let $p = 1 - p(\Par{v},v)$ where $p(\Par{v},v)$ denotes the probability that a simple random walk starting at $\Par{v}$ ever hits $v$.  If we introduce the slight variant $\ATFM^{(\lambda,p)}$ where rather than terminating particles when they hit $v_\ell$ we terminate them with probability $p$, then we see that on the tree $T(\Par{v})$, the measure $\TFM_{T(\Par{v})}^{(\lambda)}$ and $\ATFM_{T(\Par{v})}^{(\lambda,p)}$ are identical.  We also note $$\TFM_T^{(\lambda)}(v \text{ is activated}\,|\,\Par{v}\text{ is activated}) = \TFM_{T(\Par{v})}^{(\lambda)}(v \text{ is activated})\,.$$  
		By terminating all particles that hit $v_\ell$ rather than thinning out by probability $p$, we may only decrease the set of vertices activated, completing the proof. 
	\end{proof}

	\subsection{Proof of recurrence}\label{sec:jgsm3}
	
	In this section we present our main recurrence result.  The precise result consists of the following theorem.
	
	\medskip
	\begin{theorem}\label{theorem:mrrjg}
		Let $Z$ be an offspring distribution satisfying $\Prob(Z\geq 2)=1$ and $\E[Z^{4+\epsilon}]<\infty$ (for some $\epsilon>0$), and let ${\sf GW}$ be the measure on Galton-Watson trees generated by $Z$.  Then there exists a constant $\lambda_0\in(0,\infty)$ such that, for every $\lambda>\lambda_0$, the frog model with $\mathrm{Poiss}(\lambda)$ sleeping frogs at each non-root vertex is recurrent for ${\sf GW}$--a.s. every tree ${\bf T}$.  Further, $\log(\lambda_0) = O(\epsilon^{-2} + \epsilon^{-1}\log\E[Z^{4 + \epsilon}]).$
	\end{theorem}
	
	\medskip

	The main step in proving Theorem \ref{theorem:mrrjg} will consist of establishing a proposition that forms the essence of the bootstrapping argument referenced in Section \ref{sec:jgsm1} (see the second step from the sketch of the recurrence proof).  

	\begin{prop} \label{pr:iteratejg}
		
		There is a constant $C > 0$ so that the following holds.  For any $\lambda \geq C(\E[Z^4] + 1)$ and $\alpha \geq C\left(\eps^{-1} \log \E[Z^{4+\eps}] + \eps^{-2} + 1\right)$, if for all $n\geq1$ we have  
		$$\E_{{\sf AGW}_n}\left[\sum_{v'\in{\bf T}_1}{\sf HARM}_{{\bf T}}(v') 	\ATFM^{(\lambda)}_{{\bf T}}(v'\ \text{is not activated})\right]\leq e^{-\alpha}$$
		then 
		$$\E_{{\sf AGW}_n}\left[\sum_{v'\in{\bf T}_1}{\sf HARM}_{{\bf T}}(v') 	\ATFM^{(\lambda)}_{{\bf T}}(v'\ \text{is not activated})\right] \leq \frac{e^{-\alpha}}{2}\,.$$
		
	\end{prop}

	

	\begin{proof}[Proof of Theorem \ref{theorem:mrrjg}:]
		By monotonicity of the frog model with respect to $\lambda > 0$, it is sufficient to show recurrence for $\lambda_0$ defined by $\log \lambda_0 = C (\eps^{-2} + \eps^{-1} \log \E[Z^{4+\eps}])$ for sufficiently large $C$.  Additionally, by Lemma \ref{lem:truncated-frog} it is sufficient to show recurrence for the truncated frog model.   Recall that we let $v_0,v_1,\dots$ represent the vertices of a nonbacktracking path sampled according to the harmonic measure; the first step in completing the proof of the theorem will be to use Proposition \ref{pr:iteratejg} to show that for all $n\geq 1$, we have 
		\begin{equation}\label{sepvnp1ae1}\E_{{\sf AGW}\times{\sf HARM}_{{\bf T}}}\bigg[{\ATFM}^{(\lambda_0)}_{{\bf T}(v_n)}(v_{n+1}\ \text{is not activated})\bigg]=0.
		\end{equation}
		In light of Proposition \ref{pr:iteratejg}, in order to do this it will suffice to show that expectation in \eqref{sepvnp1ae1} is at most $e^{-\alpha_0}$ for every $n\geq 1$ where $\alpha_0 = C'(\eps^{-1} \log \E[Z^{4+\eps}] + \eps^{-2} )$.  Proposition \ref{pr:iteratejg} will then imply \eqref{sepvnp1ae1}.  Now using Lemma \ref{lem:lljgyhiyy}, we see that the task of establishing \eqref{sepvnp1ae1} can be further reduced to showing that
		\begin{equation}\label{lespvpajg}
			\E_{{\sf GW}}\bigg[\sum_{v'\in{\bf T}_1}{\sf HARM}_{{\bf T}}(v') {\ATFM}^{(\lambda_0)}_{{\bf T}}(v'\ \text{is not activated})\bigg]<\frac{1}{C}e^{-\alpha_0}
		\end{equation}
		since $\lambda_0$ is large enough to meet the conditions of Proposition \ref{pr:iteratejg} (where $C$ represents the universal constant appearing in Lemma \ref{lem:lljgyhiyy}).
		
		Noting that for each $v'\in{\bf T}_1$ the number of particles originating at the root of ${\bf T}^+$ that hit $v'$ is dominated by $\text{Poiss}\Big(\frac{c\lambda_0}{|{\bf T}_1|}\Big)$, we then observe that the expression inside the expectation in \eqref{lespvpajg} can be bounded above by $e^{-\frac{c\lambda_0}{|{\bf T}_1|}}$. 
		Therefore, it follows that the left-hand-side of \eqref{lespvpajg} is bounded above by
		$$\E_{{\sf GW}}\bigg[e^{-\frac{c\lambda_0}{|{\bf T}_1|}}\bigg]\leq \E_{{\sf GW}} \left[\frac{C'' |\mathbf{T}_1|}{\lambda_0} \right] = C'' \frac{\mu}{\lambda_0}$$ for some universal $C'' > 0$.  By adjusting the constant in the definition of $\lambda_0$ we see indeed have $$C'' \frac{\mu}{\lambda_0} \leq e^{-\alpha_0}$$ thus establishing \eqref{lespvpajg}; by Lemma \ref{lem:lljgyhiyy} and Proposition \ref{pr:iteratejg} we thus have established \eqref{sepvnp1ae1}.  
		
		This shows that for each $n\geq 1$ we have $$\sum_{v \in \bT_n} \HARM_{\bT}(v) {\ATFM}^{(\lambda)}_{{\bf T}(\Par{v})}(v) (\text{is not activated}) = 0 \qquad \AGW-\text{a.s.}\,.$$
		By an application of Lemma \ref{lem:augmented}, this shows \begin{equation}\label{eq:all-vertices-activated}
			\TFM_{\bT}^{(\lambda)}(v \text{ is not activated}) = 1 \quad \forall v \in \bT \quad \AGW-\text{a.s.}
		\end{equation}
		
		To show that infinitely many particles reach the root, note that 
		by Lemma \ref{lem:base}, along with the Borel-Cantelli Lemma, we see that, for $N \geq C\mu$, there is a random $n_0$ with $\AGW(n_0 < \infty) = 1$ so that for all $n \geq n_0$ we have 
		\begin{equation}\label{bnovolnwmtncjg}\sum_{v'\in{\bf T}_n}{\sf HARM}_{{\bf T}}(v') {\bf 1}_{|{\bf T}_1(v')|\geq N}\geq\frac{1}{4}
		\end{equation}
		By \eqref{eq:all-vertices-activated}, all vertices are activated, and so combining \eqref{bnovolnwmtncjg} with Lemma \ref{lem:crphmjg} shows that almost surely infinitely many particles visit the root.  Combining this with Lemma \ref{lem:truncated-frog} completes the proof.
	\end{proof}

	The upper bound on $\lambda_c$ in Theorem \ref{theorem:mrecthjg} is likely not optimal, although it is strong enough to show that recurrence depends not only on the maximum possible value of $Z$, but on the entire degree distribution.

	\begin{cor} \label{cor:max-degree-example}
		For all $d$ sufficiently large, there exists an offspring distribution $Z$ supported in $\{2,\ldots,d\}$ with $\Prob(Z = d) > 0$ and  $\lambda \in (0,\infty)$ so that the Poisson frog model with density $\lambda$ is almost-surely recurrent on Galton-Watson trees with offspring distribution $Z$, but transient on the $d$-regular tree.
	\end{cor}
	\begin{proof}
		By \cite{HJJ1}, there exists a constant $c > 0$ so that the Poisson frog model with density $\lambda$ is transient on the $d$-regular tree for $\lambda \leq c d$.  Now if we let $Z$ be the random variable $2 + (d-2)\xi$ where $\xi$ is Bernoulli with success probability $1/d^5$, then $\E[Z^5] = O(1)$.  By Theorem \ref{theorem:mrrjg}, this implies the Poisson frog model with density $\lambda$ is almost surely recurrent on Galton-Watson trees generated by $Z$ provided $\lambda \geq C$ for some $C > 0$.  Taking $d$ large enough that $c d > C$ completes the proof.
	\end{proof}

	\section{Proof of Proposition \ref{pr:iteratejg}} \label{sec:proof-prop}

	\subsection{Sketch of the proof}	\label{sec:jgsm1}
	
	The proof of Proposition \ref{pr:iteratejg} consists of a delicate bootstrapping argument; here, we isolate many of the key ideas, with the hope of providing a useful road-map through the proof.  
	
	The hypothesis of Proposition \ref{pr:iteratejg} states that \begin{equation} \label{eq:prop-hyp}
		\E_{{\sf AGW}_n}\bigg[\sum_{v'\in{\bf T}_1}{\sf HARM}_{{\bf T}}(v') 
		{\ATFM}^{(\lambda)}_{{\bf T}}(v'\ \text{is not activated})\bigg]\leq e^{-\alpha} \,.
	\end{equation}
	
	Our first main step is to show that under the assumption of \eqref{eq:prop-hyp}, a large proportion of far-away vertices will be activated.  In particular, we introduce the following ``bad'' event: 
	
	\begin{defn}
		In an instance of the truncated frog model on a tree ${T}$ and given an integer $m$ and $v \in T$, define the event  ${\mathcal{E}} = \mathcal{E}(v,m)$ via \begin{equation}
			\mathcal{E} = \left\{ \sum_{\substack{v' \in { T}_m(v) \\ v' \text{ is activated} }}{\sf HARM}_{{ T}(v)}(v')<\frac{1}{2} \right\}\,.
		\end{equation}
	\end{defn}
	
	Ultimately we will take $m = \lfloor e^{\alpha/2} \rfloor$.  Iterating \eqref{eq:prop-hyp} will show a simple upper bound on the event that $\cE$ holds for all children of the root (Lemma \ref{lem:claim1}).  We want to upgrade this bound to have probability on the order of $e^{-\alpha}$.  We will achieve this in Lemma \ref{lem:claim2} in two steps: first, \eqref{eq:prop-hyp} will show that it is rare for only one vertex in $T_1$ to be activated; second, we will bound the event that at least two distinct vertices  $v'$ and $v''$ in $T_1$ satisfy $\cE$.  By ignoring all particles originating outside of the trees $T(v')$ and $T(v'')$, we see that the probability $\cE$ occurs can only \emph{increase}.  This will allow us to upper bound the event that $\cE$ occurs for both $v'$ and $v''$ by a product of their probabilities (Claim \ref{claim-in-claim2}). 
	
	With Lemma \ref{lem:claim2} in-hand we then proceed to perform the boot-strapping step.  The game here is to win over $e^{-\alpha}$ by a multiplicative factor of $1/2$ in the right-hand of side of \eqref{eq:prop-hyp}.  Active particles on the first level are either activated by particles starting at $v$---recall that in the augmented truncated frog model there are $\Poiss(\lambda)$ particles at $v$---or particles from a different child of $v$.  To handle the latter, we will use $\cE$ to identify many activated particles at height $m$ along with Lemma \ref{lem:crphmjg} to compare the harmonic measure and return probabilities.  To handle particles activated by those starting at $v$, we either have that the first generation is small---and thus one of the $\Poiss(\lambda)$ particles is likely to land there---or the first generation is large, which is rare for Galton-Watson trees. As such, we break up the proof into the case of small $|T_1|$ and large $|T_1|$; these are handled in Lemmas \ref{lem:claim3} and \ref{lem:claim4} respectively. 
	
	The proofs of Lemmas \ref{lem:claim3} and \ref{lem:claim4} follow similar paths.  For each vertex $v' \in T_1$, we want to upper bound the probability that $v'$ is not activated.  To identify particles that will activate $v'$, we will want to condition on $\cE^c$ holding for some child of the root, which holds with probability $C e^{-\alpha}$ for our choice of $m = \lfloor e^{\alpha/2} \rfloor$.  To improve over this constant $C$ in the case when $T_1$ is small, we will use the $\Poiss(\lambda)$ particles at the root to beat this constant (see \eqref{eq:claim3-first-split}); in the case when $T_1$ is large, we will use a tail bound on the offspring distribution $Z$ to make this probability as small as we want.  To balance these two cases, our cutoff for ``small'' versus ``large'' is  at $|T_1| < \sqrt{\lambda}$.  
	
	Once we have that the event $\cE^c$ holds for some child of the root, it just remains to show that many particles visit the root.  The main tool will be Lemma \ref{lem:crphmjg}, although it is only useful when then quantities $|T_1(u)|$ and $|T_1(\Par{v})|$ are not too large.  As such, we will eliminate certain ``bad tree events'' in which the harmonic measure has a large contribution from nodes that have many children using the tools of Section \ref{sec:cutting}.  The particular notions of ``bad tree'' are slightly different in Lemmas \ref{lem:claim3} and \ref{lem:claim4}, due only to the fact that in Lemma \ref{lem:claim4} we may have a large first generation.  Once these ``bad trees'' are eliminated, a Poisson thinning argument---Claims \ref{claim-in-claim3} and \ref{claim-in-claim4} in Lemmas \ref{lem:claim3} and \ref{lem:claim4} respectively---will identify many active particles with the potential to activate $v'$.

	\subsection{Setting up the bootstrap: identifying many activated sites}

	The main goal of this subsection is to prove Lemma \ref{lem:claim2}, which bounds the event that $\cE$ holds for all first generation vertices.  We first compute an expectation, which will later be used for bounding the probability $\cE$ holds.  
	
	\begin{lemma}\label{lem:claim1}
		Suppose that for some $n, \alpha,\lambda$ we have $$\E_{{\sf AGW}_n}\bigg[\sum_{v'\in{\bf T}_1}{\sf HARM}_{{\bf T}}(v') {\ATFM}^{(\lambda)}_{{\bf T}}(v'\ \text{is not activated})\bigg]\leq e^{-\alpha}\,.$$
		Then for all $m$ we have 
		\begin{equation}\label{ineqtfmmejg}
			\E_{{\sf AGW}_n \times {\ATFM}^{(\lambda)}_{{\bf T}}}\left[\sum_{\substack{v' \in {\bf T}_m \\ v' \text{ is activated} }}{\sf HARM}_{{\bf T}}(v')\right]\geq 1-me^{-\alpha}.
		\end{equation}
	\end{lemma}
	\begin{proof}
		The proof will consist of iterating the hypothesis $m$ times.  We start by noting that for any possible tree $T$, non-backtracking path $\omega = (v_0,v_1,\ldots)$, and any pair of positive integers $n,m$, we have \begin{align*}{\ATFM}^{(\lambda)}_{T(v_n)}(v_{n+m}\ \text{is activated}) 
			&\geq\prod_{j=0}^{m-1}{\ATFM}^{(\lambda)}_{T(v_{n+j})}(v_{n+j+1}\ \text{is activated})\\
			&=\prod_{j=0}^{m-1}\left(1-\bigg(1-{\ATFM}^{(\lambda)}_{T(v_{n+j})}(v_{n+j+1}\ \text{is activated})\bigg)\right)\\
			&\geq 1-\sum_{j=0}^{m-1}\left(1-{\ATFM}^{(\lambda)}_{T(v_{n+j})}(v_{n+j+1}\ \text{is activated})\right).
		\end{align*}
		Combining this bound with the assumption proves
		\begin{align*}
			\E_{{\sf AGW}\times{\sf HARM}_{{\bf T}(v_n)}}&\bigg[{\ATFM}^{(\lambda)}_{{\bf T}(v_n)}(v_{n+m}\ \text{is activated})\bigg]\\&\geq 1-\sum_{j=0}^{m-1}\left(1-\E_{{\sf AGW}\times{\sf HARM}_{{\bf T}(v_n)}}\bigg[{\ATFM}^{(\lambda)}_{{\bf T}(v_{n+j})}(v_{n+j+1}\ \text{is activated})\bigg]\right)\\&\geq 1-m e^{-\alpha}\,.
		\end{align*}
	\end{proof}

	We are now ready to state and prove the main result of this subsection.  While this Lemma is stated for all $m \geq 1$, we will ultimately take $m = \lfloor e^{\alpha/2}\rfloor$.

	\begin{lemma}\label{lem:claim2}
		There exists a universal constant $C >0$ so that the following holds.  If for all $n \geq 1$ we have $$\E_{{\sf AGW}_n}\bigg[\sum_{v'\in{\bf T}_1}{\sf HARM}_{{\bf T}}(v') {\ATFM}^{(\lambda)}_{{\bf T}}(v'\ \text{is not activated})\bigg]\leq e^{-\alpha}$$
		then for all $m \geq 1$ we have 
		\begin{equation}\label{ineqequivcjg}
			\AGW_n \times {\ATFM}^{(\lambda)}_{{\bf T}}\left(\forall\ v\in{\bf T}_1\,, \cE(v,m) \ \text{ holds}\right)\leq  Ce^{-\alpha}\bigg(1+m^2 e^{-\alpha}\bigg).
		\end{equation}
	\end{lemma}
	\begin{proof}
		Let $T$ be generated by ${\sf AGW}_n$.  Let $v^*$ be the vertex in $T_1$ that is hit by the frog originating at the root which follows a standard loop-erased path to $\infty$ (recall that the other $\text{Poiss}(\lambda)$ frogs starting at the root move according to a slightly different set of dynamics under $\ATFM_T^{(\lambda)}$).  In addition, we let $\hat{v}$ be the vertex in $T_1\setminus\{v^*\}$ that is hit by the frog with minimal index, out of all of the frogs originating at either the root or in $T(v^*)$ that hit $T_1\setminus\{v^*\}$ (presuming any such frogs exist). 
		Then we have 
		\begin{align}\label{ubtejgjg}&{\ATFM}^{(\lambda)}_{T}\bigg(\forall\ v\in T_1, \cE(v,m) \text{ holds}\bigg)\\
			&\leq{\ATFM}^{(\lambda)}_{T}\Big(\text{Only\ }1\ \text{vertex in}\ T_1\ \text{is activated}\Big)\nonumber\\
			&\quad +{\ATFM}^{(\lambda)}_{T}\Big(\text{At least}\ 2\ \text{vertices in}\ T_1\ \text{are activated, and both}\ v^*\ \text{and}\ \hat{v}\ \text{satisfy}\ \mathcal{E}\Big) \nonumber \,.
		\end{align}
		By Markov's inequality and Lemma \ref{lem:1/2}, we may bound the first term by \begin{align} \label{eq:one-activated}
			\AGW_n \times &{\ATFM}^{(\lambda)}_{T}\Big(\text{Only\ }1\ \text{vertex in}\ T_1\ \text{is activated}\Big) \\
			&\leq {\AGW_n} \times {\ATFM}^{(\lambda)}_{T}\left(\sum_{\substack{v' \in T_1\\v' \text{ not activated}}} \HARM_T(v') \geq 1/2 \right) \nonumber \\
			&\leq 2 e^{-\alpha} \nonumber 
		\end{align}
		where we used the assumption to bound the expectation when applying Markov's inequality.

		The first step to bounding the second term in \eqref{ubtejgjg} will be in the following claim. 
		\noindent	\begin{claim} \label{claim-in-claim2}
			\begin{align*}
				&{\ATFM}^{(\lambda)}_{T}\Big(\text{At least}\ 2\ \text{vertices in}\ T_1\ \text{are activated, and both}\ v^*\ \text{and}\ \hat{v}\ \text{satisfy}\ \mathcal{E}\Big)\\
				&\leq\sum_{\substack{v',v'' \in T_1 \\ v'\neq v'' }}{\sf HARM}_T(v')\cdot\frac{{\sf HARM}_T(v'')}{1-{\sf HARM}_T(v')} {\ATFM}^{(\lambda)}_{T(v')}(\mathcal{E})\cdot {\ATFM}^{(\lambda)}_{T(v'')}(\mathcal{E})
			\end{align*}
		\end{claim}
		\noindent{\sc Proof of Claim \ref{claim-in-claim2}.}
		We first condition on the identity of $v^*$ and bound 
		\begin{align*}
			&{\ATFM}^{(\lambda)}_{T}\Big(\text{At least}\ 2\ \text{vertices in}\ T_1\ \text{are activated, and both}\ v^*\ \text{and}\ \hat{v}\ \text{satisfy}\ \mathcal{E}\Big)\\
			&\leq \sum_{\substack{v'\in T_1 }} \HARM_T(v') {\ATFM}^{(\lambda)}_{T}\Big( v' \text { satisfies }\cE, \hat{v}\ \text{satisfies}\ \mathcal{E} \,|\, v^* = v'\Big)
		\end{align*}
		where we interpret the probability on the right-hand side to be zero if no vertex in $T_1$ is activated other than $v'$; we note that the only reason this is an inequality rather than equality is due to the additional vertex added to the root of $T$ in $\ATFM_T^{(\lambda)}$.  Write $\mathcal{B}$ for the event that a particle starting in $T(v')$ ever reaches $T_1 \setminus \{v'\}$.   We may then write \begin{align*}
			{\ATFM}^{(\lambda)}_{T}\Big(& v' \text { satisfies }\cE, \hat{v}\ \text{satisfies}\ \mathcal{E} \,|\, v^* = v'\Big) \\
			&= {\ATFM}^{(\lambda)}_{T}\Big( v' \text { satisfies }\cE, \mathcal{B} \,|\, v^* = v'\Big) \\
			&\qquad \times \sum_{v'' \in T_1 \setminus \{v'\}}  {\ATFM}^{(\lambda)}_{T}\Big( v'' = \hat{v}, v''  \text { satisfies }\cE \,|\, v^* = v', \mathcal{B} \Big)\,.
		\end{align*}
		We may bound \begin{align*}
			{\ATFM}^{(\lambda)}_{T}\Big( v' \text { satisfies }\cE, \mathcal{B} \,|\, v^* = v'\Big) \leq {\ATFM}^{(\lambda)}_{T}\Big( v' \text { satisfies }\cE \,|\, v^* = v'\Big) = \ATFM^{(\lambda)}_{T(v')}(\cE)\,.
		\end{align*}
		
		Similarly bound \begin{align*} 
			{\ATFM}^{(\lambda)}_{T}\Big( v'' = \hat{v}, v''  \text { satisfies }\cE \,|\, v^* = v', \mathcal{B} \Big) = \frac{\HARM_T(v'')}{1 - \HARM_T(v')} \ATFM^{(\lambda)}_{T(v'')}(\cE)
		\end{align*}
		where we note that the first term is precisely the probability that a loop-erased random walk in $T$ starting at the root steps to $v''$ conditioned on it not stepping to $v'$.  Combining the previous four displayed equations completes the Claim. \hfill $\blacksquare$
		
		%

		By Lemma \ref{lem:1/2}, the harmonic measure on a level $1$ vertex differs from the uniform measure on $T_1$ by a multiplicative factor of at most $2$ (see the proof of Lemma \ref{lem:GW1} for the details of this argument).  We now may apply Claim \ref{claim-in-claim2} to see
		%
		\begin{align}\label{eq:E-sq}
			&{\ATFM}^{(\lambda)}_{T}\Big(\text{At least}\ 2\ \text{vertices in}\ T_1\ \text{are activated, and both}\ v^*\ \text{and}\ \hat{v}\ \text{satisfy}\ \mathcal{E}\Big)\\
			&\leq C\sum_{\substack{v',v'' \in T_1 \\ v'\neq v'' }}\frac{1}{\binom{|T_1|}{2}}\cdot{\ATFM}^{(\lambda)}_{T(v')}(\mathcal{E})\cdot{\ATFM}^{(\lambda)}_{T(v'')}(\mathcal{E})  \nonumber.  
		\end{align}
		for some constant $C>0$. 
		
		Combining \eqref{eq:one-activated} and \eqref{eq:E-sq} along with Lemma \ref{lem:lljgyhiyy} bounds

		\begin{align*}
			{\sf AGW}_n \times {\ATFM}^{(\lambda)}_{{\bf T}}\bigg(\forall\ v\in{\bf T}_1\,, \cE(v',m) \ \text{ holds}\bigg)\\
			\leq 2 e^{-\alpha} + C' \left[{\AGW_n} \times \ATFM_{\bT}^{(\lambda)}(\cE) \right]^2
		\end{align*}
		for a constant $C' > 0$.  Applying Lemma \ref{lem:claim1} and Markov's inequality bounds $${\AGW_n} \times \ATFM_{\bT}^{(\lambda)}(\cE) \leq 2m e^{-\alpha}$$
		completing the proof of the Lemma. 
	\end{proof}

	\subsection{Performing the bootstrap}
	We will break up the proof of Proposition \ref{pr:iteratejg} into two cases, depending on if the first generation is ``small'' or ``large''.  Our cutoff for these two is the event $|\bT_1| < \sqrt{\lambda}$ and its complement.  We handle the case of ``small'' $|\bT_1|$ first.

	\begin{lemma}\label{lem:claim3}
		There is a constant $C > 0$ so that the following holds.  For any $\alpha \geq 1$ and $\lambda \geq C \E[Z^4] $, if for all $n \geq 1$ we have $$\E_{{\sf AGW}_n}\bigg[\sum_{v'\in{\bf T}_1}{\sf HARM}_{{\bf T}}(v') {\ATFM}^{(\lambda)}_{{\bf T}}(v'\ \text{is not activated})\bigg]\leq e^{-\alpha}$$
		then $$\E_{{\sf AGW}_n}\bigg[\mathbf{1}_{|{\bf T}_1|<\sqrt{\lambda}}\sum_{v'\in{\bf T}_1}{\sf HARM}_{{\bf T}}(v') {\ATFM}^{(\lambda)}_{{\bf T}}(v'\ \text{is not activated})\bigg] \leq \frac{e^{-\alpha }}{4}\,.$$
	\end{lemma}
	\begin{proof}
		Let $T$ be a tree generated by $Z$ with $|T_1| < {\lambda}$ and $v' \in T_1$.  Note first that the probability none of the $\Poiss(\lambda)$ particles starting at the root of $T^+$ hit $v'$ is upper bounded by $e^{-C_1 \sqrt{\lambda}}$.  Set $m = \lfloor e^{\alpha/2}\rfloor$ and bound 
		\begin{align}&{\ATFM}^{(\lambda)}_{T}(v'\ \text{is not activated}) \label{eq:claim3-first-split}\\
			&\leq e^{-C_1\sqrt{\lambda}}{\ATFM}^{(\lambda)}_{T}(\forall ~v\in T_1(v^*), \cE(v,m)\text{ holds}) \nonumber  \\&+e^{-C_1\sqrt{\lambda}}{\ATFM}^{(\lambda)}_{T}(\exists\ v''\in T_1(v^*)\ \text{satisfying}\ \mathcal{E}^c\ \text{and nothing from}\ T(v'')\ \text{hits}\ v') \nonumber 
		\end{align}
		where we recall that $v^*$ is vertex in $T_1$ hit by the active particle that begins at the root of $T^+$.  Applying Lemma \ref{lem:claim2} along with Lemma \ref{lem:lljgyhiyy} shows a bound of \begin{equation}\label{eq:claim3-first-bound}
			\E_{{\sf AGW}_n}  e^{-C_1\sqrt{\lambda}}{\ATFM}^{(\lambda)}_{\bT}(\forall ~v\in \bT_1(v^*), \cE(v,m)\text{ holds}) \leq C e^{-C_1 \sqrt{\lambda}} e^{-\alpha}
		\end{equation}
		for some constant $C > 0$.  Moving on to the second term in \eqref{eq:claim3-first-split}, bound 
		\begin{align} \label{eq:claim3-find-one}
			\sum_{v' \in T_1} &\HARM_T(v') {\ATFM}^{(\lambda)}_{T}(\exists\ v''\in T_1(v^*)\ \text{satisfying}\ \mathcal{E}^c\ \text{and nothing from}\ T(v'')\ \text{hits}\ v') \\
			&\leq \max_{ \substack{v' \in T_1 \\ v'' \in T_2}}\ATFM_{T}^{(\lambda)}(\text{Nothing from }T(v'')\text{ hits }v'\,|\,v'' \in T_1(v^*), \cE^c(v'',m)\text{ holds}) \nonumber 
		\end{align}

		Define the event $B:=\{\exists\ v\in T_2:T(v)\in A\}\cup\{\exists\ v\in T_1:|T_1(v)|>e^{\alpha/3}\}$ where $A$ is the event from Lemma \ref{lem:A-rare}.  Since $Z$ has four moments, we see that for $N\geq C \mu$ we have 
		\begin{equation*}
			{\sf GW}(B)\leq\mu^2 C e^{-c2^{\frac{m}{2}}}+\mu \E[Z^3] e^{-\alpha} \leq C' \E[Z^4] e^{-\alpha}
		\end{equation*}
		where in the second bound we used $m = \lfloor e^{\alpha/2}\rfloor$ and the fact that for any random variable $X$ with $X \geq 1$ we have $\E X \E X^3 \leq \E X^4$\,. Applying Lemma \ref{lem:lljgyhiyy} shows \begin{equation}\label{eq:B-bound}
			\AGW_n(B)\leq C'' \E[Z^4] e^{-\alpha}
		\end{equation}
		We now bound the right-hand-side of \eqref{eq:claim3-find-one} in the case that $T \in B^c$. 
		\begin{claim}\label{claim-in-claim3}
			Let $T \in B^c$ satisfy $|T_1| < \sqrt{\lambda}$.  Then for any $v' \in T_1, v'' \in T_2$ we may bound $$\ATFM_{T}^{(\lambda)}(\text{Nothing from }T(v'')\text{ hits }v'\,|\,v'' \in T_1(v^*), \cE^c(v'',m)\text{ holds}) \leq Ce^{-\alpha}\,. $$
		\end{claim}
		\noindent{\sc Proof of Claim \ref{claim-in-claim3}.}
		Note that if $v''$ satisfies $\mathcal{E}^c$ and $T\in B^c$, then this implies that
		\begin{equation}\label{jinactabncjg}
			\sum_{\substack{v' \in T_i(v''),\ |T_1(v')|<\sqrt{\lambda} \\ v' \text{ is activated} }}{\sf HARM}_{T(v'')}(v')\geq\frac{1}{4}
		\end{equation} 
		for every $i$ with $\frac{m}{2}<i\leq m$.  In addition, since $|T_1|<\sqrt{\lambda}$, then since each vertex in $T(v'')$ begins with $\text{Poiss}(\lambda)$ inactive particles and since $T\in B^c$ implies that the parent of $v''$ has no more than $e^{\alpha/3}$ children, it must follow from Lemma \ref{lem:crphmjg} that, conditioning on \eqref{jinactabncjg} holding for each $i$ with $\frac{m}{2}<i\leq m$, the number of particles from $T(v'')$ that hit $v_0$ stochastically dominates $\text{Poiss}(\frac{C_4}{\lambda}\cdot \lambda\cdot\frac{m}{8}\cdot e^{-\alpha/3})=\text{Poiss}(C_5\cdot m\cdot e^{-\alpha/3})$ for some universal $C_5$. 
		To see this Poisson domination, note that conditioning on the collection of activated vertices does not give any information about the
		number of frogs at each vertex whose walks go to $v_0$, due to the dynamics of the truncated frog model.  Thus the probability $v'$ is not activated is at most $e^{-C_5 me^{-\alpha/3}}$.  Recalling $m = \lfloor e^{\alpha/2}\rfloor$ completes the Claim.
		\hfill $\blacksquare$

		Combining \eqref{eq:claim3-first-split}, \eqref{eq:claim3-first-bound} and \eqref{eq:claim3-find-one} with Claim \ref{claim-in-claim3} bounds \begin{align*}
			\E_{{\sf AGW}_n}&\bigg[\mathbf{1}_{|{\bf T}_1|<\sqrt{\lambda}}\sum_{v'\in{\bf T}_1}{\sf HARM}_{{\bf T}}(v') {\ATFM}^{(\lambda)}_{{\bf T}}(v'\ \text{is not activated})\bigg]\\
			&\leq C e^{-C_1 \sqrt{\lambda}} \left(e^{-\alpha} + \AGW_n(B)  \right)\\
			&\leq C' e^{-C_1 \sqrt{\lambda}}e^{-\alpha} \E[Z^4]\,.
		\end{align*}
		
		Taking $\lambda \geq C\E[Z^4]$ for some large but universal $C$ completes the claim.
	\end{proof}


	We now move on to the proof in the case of ``large'' $\bT_1$.
	\begin{lemma}\label{lem:claim4}
		There is a constant $C > 0$ so that the following holds.  For any $\alpha  \geq C(\epsilon^{-1}\log \E[Z^{4 + \epsilon}] + \epsilon^{-2})$ and $\lambda \geq C\big(\E[Z^{4 + \epsilon}]^{\frac{2}{4+\epsilon}}\big)$, if for all $n \geq 1$ we have 
		
		$$\E_{{\sf AGW}_n}\bigg[\sum_{v'\in{\bf T}_1}{\sf HARM}_{{\bf T}}(v') {\ATFM}^{(\lambda)}_{{\bf T}}(v'\ \text{is not activated})\bigg]\leq e^{-\alpha}$$
		then $$\E_{{\sf AGW}_n}\bigg[\mathbf{1}_{|{\bf T}_1|\geq \sqrt{\lambda}}\sum_{v'\in{\bf T}_1}{\sf HARM}_{{\bf T}}(v') {\ATFM}^{(\lambda)}_{{\bf T}}(v'\ \text{is not activated})\bigg] \leq \frac{e^{-\alpha }}{4}\,.$$
	\end{lemma}
	\begin{proof}
		We begin along the same lines as Lemma \ref{lem:claim3}.  Let $T$ be a tree generated by $Z$ with $|T_1| \geq \lambda$ and $v' \in T_1$.  For $m = \lfloor e^{\alpha/2} \rfloor$ we bound 	
		\begin{align}{\ATFM}^{(\lambda)}_{T}&(v'\ \text{is not activated}) \label{eq:claim4-first-split}\\
			&\leq{\ATFM}^{(\lambda)}_{T}(\forall ~v \in T_1(v^*), \cE(v,m)\text{ holds}) \nonumber  \\&+{\ATFM}^{(\lambda)}_{T}(\exists\ v''\in T_1(v^*)\ \text{satisfying}\ \mathcal{E}^c\ \text{and nothing from}\ T(v'')\ \text{hits}\ v')\,. \nonumber 
		\end{align}
		
		To bound the first term on the right-hand side of \eqref{eq:claim4-first-split}, we will use Lemma \ref{lem:claim2} along with the tree comparison statement Lemma \ref{lem:GW1}.  In particular we see \begin{align*}
			&\E_{\AGW_n}\left[\one_{|\bT_1| \geq \sqrt{\lambda}} {\ATFM}^{(\lambda)}_{\bT}(\forall ~v\in \bT_1(v^*), \cE(v,m)\text{ holds}) \right] \\
			&= \sum_{j \geq \sqrt{\lambda}} \AGW_n(|\bT_1| = j) \E_{\AGW_n}\left[ {\ATFM}^{(\lambda)}_{\bT}(\forall ~v\in \bT_1(v^*), \cE(v,m)\text{ holds})\big| |\bT_1| = j \right]   \\
			&\leq C \P(Z \geq \sqrt{\lambda}) e^{-\alpha}  
		\end{align*}
		where the last inequality follows from Lemma \ref{lem:claim2} along with the tree-comparison statements Lemma \ref{lem:GW1} and Lemma \ref{lem:lljgyhiyy}.  In particular, by Markov's inequality we have \begin{equation}\label{eq:claim4-first-bound}
			\E_{\AGW_n}\left[\one_{|\bT_1| \geq \sqrt{\lambda}} {\ATFM}^{(\lambda)}_{\bT}(\forall ~v\in \bT_1(v^*), \cE(v,m)\text{ holds}) \right] \leq C' e^{-\alpha} \frac{\E[Z^2]}{\lambda}\,.
		\end{equation}
		
		Moving on to the second term in \eqref{eq:claim4-first-split},  we proceed as in Lemma \ref{lem:claim3}.  If we let $v^*$ be the vertex that the first active particle starting at the root goes to, we again have \eqref{eq:claim3-find-one} holds: \begin{align} \label{eq:claim4-find-one}
			\sum_{v' \in T_1} &\HARM_T(v'){\ATFM}^{(\lambda)}_{T}(\exists\ v''\in T_1(v^*)\ \text{satisfying}\ \mathcal{E}^c\ \text{and nothing from}\ T(v'')\ \text{hits}\ v') \\
			&\leq \max_{ \substack{v' \in T_1 \\ v'' \in T_2}} \ATFM_{T}^{(\lambda)}(\text{Nothing from }T(v'')\text{ hits }v'\,|\,v'' \in T_1(v^*), \cE^c(v'',m)\text{ holds})\,. \nonumber 
		\end{align}
		
		We will now similarly cut out ``bad trees'' as we did in the case of Lemma \ref{lem:claim3}.  Set $R = \exp(\frac{\alpha}{4 + \eps/2})$; for the event $A$ defined in Lemma \ref{lem:A-rare}, define the event $\mathcal{A}$ via \begin{equation}\label{eq:cA-def}
			\mathcal{A} = \{\exists~v \in \bT_1(v^*) : \bT(v) \in A  \}  \cup \{\bT_1 \geq R \} \cup \{\bT_1(v^*) \geq R\}\,.
		\end{equation}
		Away from the event $\mathcal{A}$ we will bound the right-hand-side of \eqref{eq:claim4-find-one}.  
		
		\begin{claim}\label{claim-in-claim4}
			Let $T$ be a rooted tree with $u \in T_1, w \in T_2$.  Suppose that $T$ satisfies $|T_1| < R$, $|T_1(\Par{w})| < R$ and for all $v' \in T_1(\Par{w})$ we have  $T(v') \notin A$.  Then we have $$\ATFM_{T}^{(\lambda)}(\text{Nothing from }T(w)\text{ hits }u\,|\,w \in T_1(v^*), \cE^c(w,m)\text{ holds}) \leq C\exp\left(-c\sqrt{\lambda} e^{\frac{\eps \alpha}{16 + 2\eps}}  \right)$$
			for universal constants $C, c > 0$.
		\end{claim}
		\noindent{\sc Proof of Claim \ref{claim-in-claim4}.}
		Note that since $T(w) \notin A$ and $\cE^c(w,m)$ holds, we have that 
		\begin{equation}\label{eq:claim4-claim-bd}
			\sum_{\substack{v' \in T_i(w),\ |T_1(v')|<\sqrt{\lambda} \\ v' \text{ is activated} }}{\sf HARM}_{T(w)}(v')\geq\frac{1}{4}
		\end{equation} 
		for every $i$ with $\frac{m}{2}<i\leq m$.  Since $|T_1| < R$ and $|T_1(\Par{w})| < R$ and since each vertex in $T(w)$ begins with $\text{Poiss}(\lambda)$ inactive particles, it follows from Lemma \ref{lem:crphmjg} that, conditioning on \eqref{eq:claim4-claim-bd} holding for each $i$ with $\frac{m}{2}<i\leq m$, the number of particles from $T(w)$ that hit $u$ stochastically dominates $\text{Poiss}(c \sqrt{\lambda} m / R^2)$ for some universal $c > 0$.  Recalling $m = \lfloor e^{\alpha/2}\rfloor$ and $R = \exp(\frac{\alpha}{4 + \eps/2})$ completes the Claim.
		\hfill $\blacksquare$
		
		Combining \eqref{eq:claim4-first-split}, \eqref{eq:claim4-first-bound} and \eqref{eq:claim4-find-one} with Claim \ref{claim-in-claim4} bounds \begin{align*}
			\E_{{\sf AGW}_n}&\bigg[\mathbf{1}_{|{\bf T}_1|\geq \sqrt{\lambda}}\sum_{v'\in{\bf T}_1}{\sf HARM}_{{\bf T}}(v'){\ATFM}^{(\lambda)}_{{\bf T}}(v'\ \text{is not activated})\bigg] \\
			&\leq C' e^{-\alpha} \frac{\E[Z^2]}{\lambda} + C'\exp\left(-c\sqrt{\lambda} e^{\frac{\eps \alpha}{16 + 2\eps}}  \right) + {\AGW_n} \times \ATFM_{\bT}^{(\lambda)}(\mathcal{A}) 
		\end{align*}
		where $\mathcal{A}$ is the event defined in \eqref{eq:cA-def}.  We may bound \begin{align*}
			{\AGW_n} &\times \ATFM_{\bT}^{(\lambda)}(\mathcal{A}) \\
			&= \AGW_{n+1}(\exists~v \in \bT_1 : \bT(v) \in A ) + \AGW_n(|\bT_1| \geq R) + \AGW_{n+1}(|\bT_1| \geq R) \\
			&\leq C \mu \exp(-c 2^{m/2}) + C \P(Z \geq R) \\
			&\leq C' \left(\mu \exp(-c 2^{m/2}) + \E[Z^{4 + \eps}] \exp\left(-\left(1 + \frac{\eps}{8+\eps}\right)\alpha\right)\right)\,.
		\end{align*}
		Choosing $C$ sufficiently large for $\lambda \geq C \E[Z^2]$ we have $$ C' e^{-\alpha} \frac{\E[Z^2]}{\lambda} \leq \frac{e^{-\alpha}}{16}\,.$$
		For $\alpha \geq C (\eps^{-2} + 1)$ with $C$ large enough we have $e^{\frac{\eps \alpha}{16 + 2\eps}} \geq \alpha$ and so $$C'\exp\left(-c\sqrt{\lambda} e^{\frac{\eps \alpha}{16 + 2\eps}}  \right) \leq \frac{e^{-\alpha}}{16}\,. $$
		For  $\alpha  \geq C(\epsilon^{-1}\log \E[Z^{4 + \epsilon}] +  1)$ we may take $C$ large enough so that $$C'\E[Z^{4 + \eps}] \exp\left(-\left(1 + \frac{\eps}{8+\eps}\right)\alpha\right) \leq \frac{e^{-\alpha}}{16}\,.$$
		Finally since $m = \lfloor e^{\alpha/3} \rfloor$ for $\alpha \geq C( \log \mu  + 1 )$ we have $$C' \mu \exp(-c e^{m/2}) \leq \frac{e^{-\alpha}}{16}\,.$$
		Combining the previous displayed equations completes the proof.
	\end{proof}

	\begin{proof}[Proof of Proposition \ref{pr:iteratejg}]
		The proof follows from combining Lemmas \ref{lem:claim3} and \ref{lem:claim4}.
	\end{proof}

	\section{Counterexamples and open questions} \label{sec:questions}
	
	In this section we give an example of a tree that does not have a recurrent regime for the Poisson frog model.  In addition, we also provide an example of a tree for which the Poisson frog model has a non-trivial intermediate regime between recurrence and transience.  We conclude the section by discussing some remaining open problems.
	
	\subsection{A tree without a recurrent regime} 
	Let $\hat{T}$ be the rooted tree for which each vertex on level $n$ has $n+2$ children.  Using methods similar to those employed in the proof of Theorem \ref{th:transiencemd}, we will now establish the following transience result on $\hat{T}$.
	
	\begin{lemma} \label{lem:rtnrr} 
		The frog model on $\hat{T}$ with i.i.d. $\text{Poiss}(\lambda)$ frogs per non-root vertex is transient for every $\lambda>0$.
	\end{lemma}
	
	\begin{proof}
		We begin by defining the following branching model on $\hat{T}$ which dominates the frog model with i.i.d. $\text{Poiss}(\lambda)$ frogs per non-root vertex with respect to the number of returns to the root.  To start, we first select a positive integer $N$ that is large enough so that $\frac{N^2}{4(N+1)}>\lambda$.  We then assign i.i.d. $\text{Poiss}(\lambda)$ active particles to each non-root vertex on every level $n<N$, along with a single active particle at the root.  In addition, any time a particle takes a step away from the root and lands on a vertex $v$ for which $|v|\geq N$, it gives birth to $\text{Poiss}(\lambda)$ additional active particles at that vertex.  Next we define $\alpha:=\Big((\lambda+1)(N+1)\Big)^{-1/2}$ and the function $w:\mathbb{N}\to\mathbb{R}$ as 
		$$w(j)=\begin{cases}\Big(\frac{(j+2)!}{2}\Big)^{-1/2} &\text{if }j<N\\
			\Big(\frac{(N+1)!}{2}\Big)^{-1/2}\alpha^{j-N+1} &\text{otherwise}\end{cases}$$
		Now once again letting $F_n$ denote the set of active particles at time $n$, and for every $f\in F_n$ denoting its distance from the root as $|f|$, we define the weight function $$W_n:=\sum_{f\in F_n}w(|f|).$$
		
		Letting $f$ represent an active particle at level $j\geq N$ at some time $n$, we see from the formulas for $w$ and $W_n$ above that the expected contribution to $W_{n+1}$ by $f$, along with any progeny it has that are born at time $n+1$, is\begin{align}\label{smibt}\frac{1}{j+3}w(j-1)+\frac{j+2}{j+3}(\lambda+1)w(j+1)&=w(j)\bigg(\frac{1}{j+3}\alpha^{-1}+\frac{j+2}{j+3}(\lambda+1)\alpha\bigg)\\&\leq w(j)\bigg(\frac{1}{N+2}\alpha^{-1}+\frac{N+1}{N+2}(\lambda+1)\alpha\bigg)\nonumber\\&=w(j)\cdot\frac{2\alpha^{-1}}{N+2}\nonumber\\&<w(j)\nonumber\end{align}(where the string of inequalities follows from \eqref{cewnp1nr}, along with the fact that $\frac{N^2}{4(N+1)}>\lambda$, $\alpha=\big((\lambda+1)(N+1)\big)^{-1/2}$, and $j+3>N+2$).  If instead we have $|f|=N-1$ at time $n$, then the expected contribution of $f$ and its progeny at time $n+1$ will be

		\begin{align}\label{cenm1c}
			\frac{1}{N+2}w(N-2)&+\frac{N+1}{N+2}(\lambda+1)w(N) \\
			&=w(N-1)\bigg(\frac{1}{N+2}(N+1)^{1/2}+\frac{N+1}{N+2}(\lambda+1)\alpha\bigg)\nonumber \\&=w(N-1)\bigg(\frac{1}{N+2}\alpha^{-1}(\lambda+1)^{-1/2}+\frac{N+1}{N+2}(\lambda+1)\alpha\bigg)\nonumber\\&<w(N-1)\bigg(\frac{1}{N+2}\alpha^{-1}+\frac{N+1}{N+2}(\lambda+1)\alpha\bigg)\nonumber\\&=w(N-1)\cdot\frac{2\alpha^{-1}}{N+2}.\nonumber
		\end{align}Likewise, in the case where $f$ is on level $j$ at time $n$ with $1\leq j<N-1$, the expected contribution to $W_{n+1}$ made by $f$ (note $f$ has no progeny at time $n+1$) is\begin{align}\label{celtnm1nr}\frac{1}{j+3}w(j-1)+\frac{j+2}{j+3}w(j+1)&=w(j)\bigg(\frac{1}{j+3}(j+2)^{1/2}+\frac{j+2}{j+3}(j+3)^{-1/2}\bigg)\\&<w(j)\cdot\frac{2}{(j+3)^{1/2}}\nonumber\\&\leq w(j).\nonumber\end{align}Finally, if $f$ is located at the root at time $n$, then its expected contribution to $W_{n+1}$ is \begin{equation}\label{cecfat}w(1)=3^{-1/2}\cdot w(0).\end{equation}
		Now setting $m=\text{max}\{\frac{2}{\sqrt{5}},\frac{2\alpha^{-1}}{N+2}\}$ (note that $\frac{2}{\sqrt{5}}$ is the value we get by plugging $j=2$ into the expression that multiplies $w(j)$ on the second line of {\RED\eqref{celtnm1nr}}), we see that it follows from \eqref{smibt}--\eqref{cecfat} that if we sum over all $f\in F_n$, then we get $${\bf E}[W_{n+1}|W_n]\leq m\cdot W_n.$$Hence, this means that $\frac{W_n}{m^n}$ is a nonnegative super-martingale, which means it converges almost surely.  Since $m<1$, this then implies that $W_n\to 0$ almost surely, thus establishing transience of our branching model on $\hat{T}$ and completing the proof.
	\end{proof}
	
	\subsection{A tree without a {\BLUE $0$-$1$} law}\label{ss:no01}

	Here we provide an example of a rooted tree with no leaves or pipes that has a non-trivial intermediate phase, meaning recurrence occurs with probability strictly between $0$ and $1$.  While \cite{HJJ2} also provides an example of a tree without a {\BLUE $0$-$1$} law, their example contains an arbitrarily long path of vertices with $1$ child each.  We therefore adapt their example for the case of a tree that is an instance of the Galton-Watson measures we've been working with (keep in mind, however, that any given tree occurs with probability $0$, and in fact Theorem \ref{th:0-1-GW} states that the collection of trees with an intermediate regime has Galton-Watson measure zero).  
	
	Our example is presented in the form of the following lemma, in which $\lambda_1(T)$ and $\lambda_2(T)$ are used to refer to $\text{sup}\{\lambda:{\sf FM}^{(\lambda)}_T(\text{transience})=1\}$ and $\text{inf}\{\lambda:{\sf FM}^{(\lambda)}_T(\text{recurrence})=1\}$ respectively.
	
	\begin{lemma} \label{lem:nontrivial-intermediate}
		Let $T$ denote the rooted tree formed by joining each of the roots of the $2$-ary tree and $d$-ary tree to a single root by a pair of distinct edges.  For $d$ sufficiently large, $\lambda_1(T) < \lambda_2(T)$.
	\end{lemma}
	\begin{proof}
		After establishing in \cite{HJJ1} the existence of both recurrent and transient regimes for the frog model on regular trees, Hoffman, Johnson, and Junge were able to conclude, by virtue of a $0$-$1$ law which they proved in \cite{HJJ2}, that $\lambda_1=\lambda_2$ on the regular $d$-ary tree (hence, they simply refer to a single critical value that we call $\lambda_c(d)$).  As shown in \cite{HJJ1}, $\lambda_c(d) \to \infty$ as $d \to \infty$, and so for $d$ sufficiently large we have $\lambda_c(d) > \lambda_c(2)$.  Now arguing as in Lemma \ref{agwasrigwasr}, for each $\lambda$ that is strictly between $\lambda_c(2)$ and $\lambda_c(d)$, there is a positive probability that for the Poisson frog model on the $d$-ary tree with $\Poiss(\lambda)$ frogs per non-root vertex, no particles ever return to the root.  This means that for such a value of $\lambda$, there is positive probability that no particles visit the root of $T$.  Conversely, by a similar argument we also know that whenever the frog beginning at the root escapes inside the $2$-ary subtree (an event with positive probability) there will be infinitely many returns to the root almost surely.  For if this were not the case, then there would be positive probability of zero returns to the root when the frog starting at the root escapes inside of the $2$-ary subtree, which would then imply that the same would hold for the frog model on the $2$-ary tree itself, thus contradicting our assumption that $\lambda>\lambda_c(2)$.  Thus we have $\lambda_1(T)\leq\lambda_c(2)<\lambda_c(d)\leq\lambda_2(T)$.
	\end{proof}

	\subsection{Further questions} Our proof of recurrence in Section \ref{sec:proofmain} (see Theorem \ref{theorem:mrrjg}) relied on the offspring distribution $Z$ having more than four moments.  It seems highly unlikely however that these are the best possible conditions, and it even seems conceivable that the recurrence result could potentially be established without imposing any moment conditions at all.  Likewise, we also were not able to extend our recurrence or our transience results to offspring distributions that can take values less than $2$, due to the presence of arbitrarily long pipes.  
	
	\begin{question}
		Does there exist a transient regime for the frog model on supercritical Galton-Watson trees in the case of $\Prob(Z \leq 1) > 0$?
	\end{question}
	
	After posting this work on arXiv, the work \cite{MW} provided a positive partial answer to this question; here they introduce two parameters, $d_{\min} = \min\{j \geq 2: \P(Z  = j) > 0\}$ and $d_{\max} = \sup\{j : \P(Z = j) > 0\}$ and showed that if $d_{\min}$ is large enough as a function of the two values $\P(Z = 0)$ and $\P(Z = 1)$, then if the mean of the density of particles is smaller than a function of $\P(Z = 0), \P(Z = 1)$ and $d_{\max}$, then there is transience.  We conjecture that for any offspring distribution $Z$ with $\E Z > 1$, there is a transient regime.
	
	In general, random walk on trees where vertices may have $1$ child may be recurrent---such as on $\mathbb{Z}$.  On supercritical Galton-Watson trees, however, not only are random walks transient, but the speed is almost surely positive, so the existence of a transient regime is still very plausible. 
	Also still on the table is expanding the cases of the offspring distribution $Z$ for which a recurrent regime exists:
	
	\begin{question}
		For which offspring distributions does there exist a recurrent regime?  In particular, what about the case of $\Prob(Z \leq 1) > 0$ or when $Z$ has fewer than four moments.  
	\end{question}
	
	While there is no obvious monotonicity with respect to degree, it may be the case that there is some monotonicity lurking.  
	\begin{question}
		Is there some stochastic order $\leq$ so that if $Z_1$ and $Z_2$ are offspring distributions with $Z_1 \leq Z_2$ then $\lambda_c(Z_1) \leq \lambda_c(Z_2)$?
	\end{question}
	
	Finally, we suspect that $\lambda_c$ is a function of more than just the mean, but our bound on $\lambda_c$ cannot rule this out.  \begin{question}
		Does there exist a pair of random variables $Z_1$ and $Z_2$ with $Z_j \geq 2, \E Z_j^5 < \infty$ for $j \in \{1,2\}$ so that $\E Z_1 = \E Z_2$ but $\lambda_c(Z_1) \neq \lambda_c(Z_2)$?  
	\end{question}

	\begin{appendix}
		
		\section{Harmonic measure and return probability} \label{app:hitting} 
		
		In this section we tie up the loose ends of Section \ref{sec:SRW}.
		\begin{proof}[Proof of Lemma \ref{lem:crphmjg}]
			We begin by defining the following quantities: First, let $\tilde{p}(v,u)$ represent the probability that simple random walk on $T(v)$ beginning at $v$ ever hits $u$.  In addition, we define $\tilde{p}(u,\infty)$ to be the probability that random walk on $T(v)$ beginning at $u$ eventually escapes through one of the children of $u$.  Turning to random walk on $T$, we let $p(v,u)$ be the probability a simple random walk starting at $v$ ever hits $u$, and we define $p(v,-\infty)$ to be the probability that random walk beginning at $v$ eventually escapes through one of the children of the root other than the parent of $v$.  Now noting that ${\sf HARM}_{T(v)}(u)=\tilde{p}(v,u)\cdot\tilde{p}(u,\infty)$ and $p_0(u)=p(u,v)\cdot p(v,-\infty)$, we see that in order to complete the proof it will suffice to show that each of the two parts of the product expression \begin{equation}\label{bndpblwuod}\bigg(\frac{p(u,v)}{\tilde{p}(v,u)}\cdot\frac{|T_1(u)|}{|T_1(v)|}\bigg)\times\bigg(p(v,-\infty)\cdot|T_1(v)|\cdot|T_1(\Par{v})|\bigg)\end{equation} are bounded away from $0$.
			
			Looking first at the quantity $\frac{p(u,v)}{\tilde{p}(v,u)}$, we define $p^*(u,v)$ and $\tilde{p}^*(v,u)$ to be the probabilities that random walk on $T(v)$ beginning at $u$ ($v$ respectively) reaches $v$ ($u$ respectively) without first returning to its starting position, and note that
			\begin{equation}\label{inbnd1t1}\frac{p(u,v)}{\tilde{p}(v,u)}\geq\frac{p^*(u,v)}{\tilde{p}^*(v,u)}\cdot\frac{\tilde{p}^*(v,u)}{\tilde{p}(v,u)},
			\end{equation}
			and
			\begin{equation}\label{inbnd1t2}
				\frac{p^*(u,v)}{\tilde{p}^*(v,u)}=\frac{|T_1(v)|}{|T_1(u)|+1}.\end{equation}
			Now we let $p$ represent the probability that random walk on $T(v)$ beginning at $v$ ever returns to $v$, and let $p'$ represent the probability that random walk on $T(v)$ beginning at $v$ returns to $v$ without first hitting $u$.  Observing that $\tilde{p}(v,u)=\frac{\tilde{p}^*(v,u)}{1-p'}\leq\frac{\tilde{p}^*(v,u)}{1-p}$, and noting that the fact that each vertex of $T$ has at least two children implies that $p\leq\frac{1}{2}$, we see that, along with \eqref{inbnd1t1} and \eqref{inbnd1t2}, this implies that $$\frac{p(u,v)}{\tilde{p}(v,u)}\cdot\frac{|T_1(u)|}{|T_1(v)|}\geq\frac{|T_1(v)|}{|T_1(u)|+1}\cdot\frac{|T_1(u)|}{|T_1(v)|}\cdot(1-p)=\frac{|T_1(u)|}{|T_1(u)|+1}\cdot(1-p)\geq\frac{2}{3}\cdot\frac{1}{2}=\frac{1}{3}.$$ Likewise, for the second part of the product in \eqref{bndpblwuod}, we see that \begin{align*}
				p(v,-\infty)\cdot|T_1(v)|\cdot|T_1(\Par{v})|&\geq\bigg(\frac{1}{|T_1(v)|+1}\cdot\frac{1}{|T_1(\Par{v})|+1}\cdot\frac{|T_1|-1}{|T_1|}\cdot\frac{1}{2}\bigg)\\
				&\qquad \times\bigg(|T_1(v)|\cdot|T_1(\Par{v})|\bigg)\\
				&=\frac{|T_1(v)|}{|T_1(v)|+1}\cdot\frac{|T_1(\Par{v})|}{|T_1(\Par{v})|+1}\cdot\frac{|T_1|-1}{|T_1|}\cdot\frac{1}{2}\\&\geq\Big(\frac{2}{3}\Big)^3\cdot\frac{1}{2}=\frac{4}{27}\end{align*}(where both inequalities follow from the fact that each vertex has degree at least $3$), thus completing the proof of the lemma.
		\end{proof}

		\begin{proof}[Proof of Lemma \ref{lem:lljgyhi}]
			We'll start by looking at the case where each vertex of $T$ has degree at least three (thus excluding the case where the root has exactly two children).  Since every vertex in $T$ has degree at least three we know that the probability that simple random walk, upon hitting $v$, ever returns to $\Par{v}$ is equal to at most $\frac{1}{2}$.  Hence, we see that simple random walk on $T$ (starting at the root) escapes through $v$ with probability at least $\frac{1}{2}f(v)$, thus establishing the upper bound in \eqref{fboundbh} for $C=2$.  To establish the lower bound, we start by letting $v_0,v_1,\dots, v_n$ represent the vertices of the nonbacktracking path beginning with the root and ending with $v$.  Now using $p(v',B)$ to denote the probability that simple random walk on $T$ beginning at a vertex $v'$ ever hits some collection of vertices $B$, and letting $p(v',B,B')$ (where $B$ and $B'$ represent disjoint collections of vertices in $T$) refer to the probability that simple random walk beginning at $v'$ eventually hits $B$ without first hitting $B'$, we find that
			\begin{align}\label{hmfh}p(v_0,v)-f(v)&\leq\sum_{j=0}^{n-1}p(v_0,v_j)\cdot p\Big(v_j,\{T_{n-j}(v_j)\setminus T_{n-j-1}(v_{j+1})\},v\big)\\ 
				&\qquad \times\Big[\underset{v'\in T_{n-j}(v_j)}{\text{max}}p(v',v_j)\Big]\cdot p(v_j,v) \noindent\\
				&\leq\sum_{j=0}^{n-1}p(v_0,v_j)\cdot p\Big(v_j,\{T_{n-j}(v_j)\setminus T_{n-j-1}(v_{j+1})\},v\Big)\cdot\Big(\frac{1}{2}\Big)^{n-j}\cdot p(v_j,v)\nonumber\\
				&=\sum_{j=0}^{n-1}p(v_0,v)\cdot p\Big(v_j,\{T_{n-j}(v_j)\setminus T_{n-j-1}(v_{j+1})\},v\Big)\cdot\Big(\frac{1}{2}\Big)^{n-j}\nonumber
			\end{align}
			(where the inequality on the second line again follows from the fact that all vertices have degree at least three).  Using the inequalities in \eqref{hmfh} we see that if $n\leq 2$, then $p(v_0,v)-f(v)\leq\frac{3}{4}p(v_0,v)$, thus implying that $f(v)\geq\frac{1}{4}p(v_0,v)$.  If $n>2$ then we have to do a little bit more work.  First, we set $m:=\text{deg}(v_{n-2})$ and let $u_1,\dots,u_{m-2}$ be the children of $v_{n-2}$ (other than $v_{n-1}$).  In addition, for each $u_j$ we define $r_j:=\sum_{v'\in T_1(u_j)}\frac{p(v',v_{n-2})}{\text{deg}(u_j)}$.  Now once again using \eqref{hmfh}, we can achieve the bound
			\begin{align}\label{bndpmft}&p(v_0,v)-f(v)\leq\sum_{j=0}^{n-3}p(v_0,v)\cdot p\Big(v_j,\{T_{n-j}(v_j)\setminus T_{n-j-1}(v_{j+1})\},v\Big)\cdot\Big(\frac{1}{2}\Big)^{n-j}\\&+p(v_0,v)\cdot p\Big(v_{n-2},\{T_2(v_{n-2})\setminus T_1(v_{n-1})\},v\Big)\cdot\frac{\sum_{j=1}^{m-2}r_j\cdot\frac{\text{deg}(u_j)-1}{\text{deg}(u_j)}}{\sum_{i=1}^{m-2}\frac{\text{deg}(u_i)-1}{\text{deg}(u_i)}}\nonumber\\&+p(v_0,v)\cdot\frac{p\Big(v_{n-2},v_{n-1},\{T_2(v_{n-2})\setminus T_1(v_{n-1})\}\Big)}{p(v_{n-2},v_{n-1})}\cdot p\Big(v_{n-1},\{T_1(v_{n-1})\setminus T_0(v)\},v\Big)\cdot\frac{1}{2}\nonumber\\&\leq\frac{p(v_0,v)}{4}+p(v_0,v)\cdot p\Big(v_{n-2},\{T_2(v_{n-2})\setminus T_1(v_{n-1})\}\Big)\cdot\frac{\sum_{j=1}^{m-2}r_j\cdot\frac{\text{deg}(u_j)-1}{\text{deg}(u_j)}}{\sum_{i=1}^{m-2}\frac{\text{deg}(u_i)-1}{\text{deg}(u_i)}}\nonumber\\&+\frac{p(v_0,v)}{2}\cdot\frac{p\Big(v_{n-2},v_{n-1},\{T_2(v_{n-2})\setminus T_1(v_{n-1})\}\Big)}{p(v_{n-2},v_{n-1})}\nonumber\end{align}
			(where the ratio of sums on the second line represents the probability that, conditioned on hitting $\{T_2(v_{n-2})\setminus T_1(v_{n-1})\}$, simple random walk eventually returns to $v_{n-2}$).  For each $u_j$, we now set $s_j:=\frac{\sum_{v'\in T_1(u_j)}p(v',u_j)}{\text{deg}(u_j)-1}$ (the probability, conditioned on hitting $T_1(u_j)$, that simple random walk ever returns to $u_j$), and note that $$r_j=\frac{s_j}{\text{deg}(u_j)}+\frac{\text{deg}(u_j)-1}{\text{deg}(u_j)}s_j r_j\implies r_j=\frac{s_j}{\text{deg}(u_j)-(\text{deg}(u_j)-1)s_j}.$$Letting $c_1$ and $c_2$ represent $p(v_{n-3},v_{n-2})$ and $p(v_{n-1},v_{n-2})$ respectively, we now see that $$p\Big(v_{n-2},\{T_2(v_{n-2})\setminus T_1(v_{n-1})\}\Big)=\frac{\sum_{j=1}^{m-2}\frac{\text{deg}(u_j)-1}{\text{deg}(u_j)}}{\text{deg}(v_{n-2})-c_1-c_2-\sum_{i=1}^{m-2}\frac{1}{\text{deg}(u_i)}}.$$ Multiplying this last expression by the ratio of sums on the second to last line in \eqref{bndpmft}, and using the above expression for $r_j$ in terms of $s_j$, we get
			\begin{equation}\label{ne1pwnn}\frac{\sum_{j=1}^{m-2}\frac{s_j}{\text{deg}(u_j)-(\text{deg}(u_j)-1)s_j}\cdot\frac{\text{deg}(u_j)-1}{\text{deg}(u_j)}}{\text{deg}(v_{n-2})-c_1-c_2-\sum_{i=1}^{m-2}\frac{1}{\text{deg}(u_i)}}\end{equation}(note that this last expression is bounded above by $\frac{1}{4}$ on account of the fact that $r_j\leq\frac{1}{4}$ for each $j$).  Now looking at the second part of the product on the last line of \eqref{bndpmft}, we see that the numerator is equal to \begin{equation}\label{bndnn2ndt}\frac{1}{\text{deg}(v_{n-2})-c_1-\sum_{j=1}^{m-2}\frac{1}{\text{deg}(u_j)}}\end{equation}
			and the denominator is equal to \begin{equation}\label{bndnndenom2p}\frac{1}{\text{deg}(v_{n-2})-c_1-\sum_{i=1}^{m-2}\frac{1}{\text{deg}(u_i)}-\sum_{j=1}^{m-2}\frac{s_j}{\text{deg}(u_j)-(\text{deg}(u_j)-1)s_j}\cdot\frac{\text{deg}(u_j)-1}{\text{deg}(u_j)}}.\end{equation}If we now let $A_1$ and $A_2$ denote the first and second sums respectively in the denominator of \eqref{bndnndenom2p}, and then plug the expressions in \eqref{ne1pwnn}, \eqref{bndnn2ndt}, and \eqref{bndnndenom2p} into the expression to the right of the final inequality in \eqref{bndpmft}, we get the inequality
			\begin{align}\label{simpvwnnjj}
				&p(v_0,v)-f(v)\\ 
				&\leq p(v_0,v)\cdot\bigg(\frac{1}{4}+\frac{A_2}{\text{deg}(v_{n-2})-c_1-c_2-A_1}+\frac{1}{2}\cdot\frac{\text{deg}(v_{n-2})-c_1-A_1-A_2}{\text{deg}(v_{n-2})-c_1-A_1}\bigg) \nonumber\\
				&=p(v_0,v)\cdot\bigg(\frac{3}{4}+\frac{A_2}{\text{deg}(v_{n-2})-c_1-c_2-A_1}-\frac{1}{2}\cdot\frac{A_2}{\text{deg}(v_{n-2})-c_1-A_1}\bigg)\nonumber\\&=p(v_0,v)\cdot\bigg(\frac{3}{4}+\frac{1}{2}\cdot\frac{A_2}{\text{deg}(v_{n-2})-c_1-c_2-A_1}\cdot\Big(1+\frac{c_2}{\text{deg}(v_{n-2})-c_1-A_1}\Big)\bigg).\nonumber\end{align}
			Next observe that, since $\frac{A_2}{\text{deg}(v_{n-2})-c_1-c_2-A_1}$ is equal to the expression in \eqref{ne1pwnn}, this means it is bounded above by $\frac{1}{4}$.  In addition, since $c_1$ and $c_2$ are each bounded above by $\frac{1}{2}$ (recall that they're return probabilities), and since all vertices of $T$ have degree at least $3$, it follows that $$\frac{c_2}{\text{deg}(v_{n-2})-c_1-A_1}\leq\frac{1/2}{m-\frac{1}{2}-\frac{m-2}{3}}=\frac{1/2}{\frac{2m}{3}+\frac{1}{6}}\leq\frac{3}{13}$$(where the last inequality follows from plugging in $m=3$).  Substituting the values of $\frac{1}{4}$ and $\frac{3}{13}$ for the two corresponding rational expressions on the last line of \eqref{simpvwnnjj}, we now find that \begin{equation}\label{fbndpmfjj}p(v_0,v)-f(v)\leq p(v_0,v)\cdot\Big(\frac{3}{4}+\frac{1}{2}\cdot\frac{1}{4}\cdot\frac{16}{13}\Big)=\frac{47}{52}\implies f(v)\geq\frac{5}{52}\cdot p(v_0,v).\end{equation}Combining this with the fact that $f(v)\geq\frac{1}{4}\cdot p(v_0,v)$ for $n\leq 2$, and the fact that ${\sf HARM}_T(v)\leq p(v_0,v)$, we can now conclude that the inequality on the left in \eqref{fboundbh} must hold for $C=11$.  Alongside the upper bound in \eqref{fboundbh} that we established for $C=2$ (and therefore $C=11$ as well), this establishes \eqref{fboundbh} for the case where each vertex of $T$ has degree at least three.
			
			To address the case where the root of $T$ has only two children, we let $T^*$ be the tree we obtain by attaching the root of the binary tree $\mathbb{T}_2$ to the root of $T$ with an edge (where the root of $T^*$ is defined to be the root vertex of $T$).  Now let $v$ be a level $n$ vertex of $T$ (with $n\geq 1$).  Since every non-root vertex of $T$ has at least two children, it follows that the probability that simple random walk starting at the root of $T^*$ escapes through one of the two level $1$ verteices of $T$ (rather than the third level $1$ vertex that was added to $T$ in order to obtain $T^*$) is at least $\frac{2}{3}$.  Hence, from this we can conclude that \begin{equation}\label{flbfcrh2c}{\sf HARM}_{T^*}(v)\leq{\sf HARM}_T(v)\leq\frac{3}{2}{\sf HARM}_{T^*}(v).\end{equation}Similarly the fact that each non-root vertex in $T$ has at least two children also implies that the probability that the first level $n$ vertex hit by simple random walk starting at the root of $T^*$ is in $T$, is at least $\frac{2}{3}$.  Thus we can also conclude that \begin{equation}\label{flbfcrh2c2}f_{T^*}(v)\leq f_T(v)\leq\frac{3}{2}f_{T^*}(v)\end{equation}(where the subscripts $T$ and $T^*$ indicate which of the two trees we are using to calculate $f$).  Since we know \eqref{fboundbh} must apply for $T^*$, it now follows from \eqref{flbfcrh2c} and \eqref{flbfcrh2c2} that it applies for $T$ as well (for $C=\frac{3}{2}\cdot 11=\frac{33}{2}$), thus completing the proof of the lemma.
		\end{proof}
		
		\section{Comparing measures on the space of trees}\label{app:comp}
		We now tie up the loose ends of Section \ref{sec:comp}

		\begin{proof}[Proof of Lemma \ref{lem:lljgyhiyy}]
			It will suffice to show that there exists $C\in(1,\infty)$ such that, for any $n\geq 1$ and any event $A$ in the space of rooted trees for which ${\sf GW}(A)>0$, we have \begin{equation}\label{2sbagwna}\frac{{\sf GW}(A)}{C}\leq{\sf AGW}_n(A)\leq C\cdot{\sf GW}(A).\end{equation}Letting $\hat{v}$ represent the first level $n$ vertex hit by simple random walk (starting at the root) on the random tree ${\bf T}$ selected with respect to the measure ${\sf AGW}$, we observe that ${\bf T}(\hat{v})$ has distribution ${\sf GW}$.  Hence, it follows that \begin{align}
				\label{2sbagwjj}{\sf GW}(A)=\int\sum_{v'\in{\bf T}_n}{\bf 1}_{{\bf T}(v')\in A}f(v')d{\sf AGW}.
			\end{align}
			Combining this with the previous lemma, we now see that 
			\begin{align}\label{1sbndjg}{\sf AGW}_n(A)&=\int\sum_{v'\in{\bf T}_n}{\bf 1}_{{\bf T}(v')\in A}{\sf HARM}_{{\bf T}}(v')d{\sf AGW} \\
				&\leq C\int\sum_{v'\in{\bf T}_n}{\bf 1}_{{\bf T}(v')\in A}f(v')d{\sf AGW}=C\cdot{\sf GW}(A)\nonumber 
			\end{align}
			(where $C$ is the constant from Lemma \ref{lem:lljgyhi}).  Likewise, if we replace $C$ by $\frac{1}{C}$ then it follows from Lemma \ref{lem:lljgyhi} that the inequality in \eqref{1sbndjg} holds in the other direction.  Hence, the proof is complete.
		\end{proof}

		\begin{proof}[Proof of Lemma \ref{lem:GW1}]
			Let $v^{(1)},\dots,v^{(j)}$ be the vertices of ${\bf T}_1$ and note that for each $i\leq j$ we have $$\frac{1-p(v^{(i)},{\bf 0})}{|{\bf T}_1|}\leq {\sf HARM}_{{\bf T}}(v^{(i)})\leq p({\bf 0},v^{(i)})\,.$$
			By Lemma\ref{lem:1/2} we have $\frac{1-p(v^{(i)},{\bf 0})}{|{\bf T}_1|} \geq \frac{1}{2|\bT_1|}$.  For the upper bound, write \begin{align*} p({\bf 0},v^{(i)}) &\leq \sum_{k = 0}^\infty \P(\text{Simple random walk first steps to } v^{(i)} \text{ from }{\bf 0} \text{ after }k \text{ returns to } {\bf 0}) \\
				&\leq \frac{1}{|\bT_1|}\sum_{k = 0}^\infty 2^{-k}  = \frac{2}{|\bT_1|}
			\end{align*}
			where the second inequality is via Lemma \ref{lem:1/2}.

			We thus have  $$\frac{1}{2|{\bf T}_1|}\leq{\sf HARM}_{{\bf T}}(v_i)\leq\frac{2}{|{\bf T}_1|}.$$Combining this with the fact that, if we choose a vertex $v'$ uniformly at random from ${\bf T}_1$ (while conditioning on $|{\bf T}_1|=j$) then ${\bf T}(v')$ has law ${\sf GW}$, we can conclude that
			\begin{equation*}
				\frac{1}{2}\leq\frac{d{\sf GW}^{(j)}_1}{d{\sf GW}}\leq 2.  \hfill \qedhere
			\end{equation*}
		\end{proof}
		
	\end{appendix}

	\section*{Acknowledgments}
		The authors are grateful to the anonymous referees for their numerous comments which helped streamline and improve the draft.		First-named author supported in part by NSF grants DMS-2137623 and DMS-2246624.  Second-named author supported by a Zuckerman STEM Postdoctoral Fellowship, as well as by ISF grant 1207/15, and ERC
		starting grant 676970 RANDGEOM.



\bibliographystyle{plain} 
\bibliography{frogGW.bib}       


\end{document}